\documentclass[10pt,a4paper]{amsart}

\usepackage{amsmath,amssymb,latexsym, mathtools}
\usepackage{amscd,amssymb,amsopn,amsmath,amsthm,graphics,amsfonts,mathrsfs,accents,enumerate,verbatim,calc}
\usepackage[dvips]{graphicx}
\usepackage{amscd}
\usepackage{amsmath}
\usepackage{amssymb}
\usepackage{amsthm}
\usepackage{tikz}
\usepackage{tikz-cd}
\usepackage[all]{xy}
\usepackage{comment}
\usepackage{docmute}
\usepackage{cleveref}

\theoremstyle{plain}
\newtheorem{thm}{Theorem}[section]
\newtheorem{prop}[thm]{Proposition}
\newtheorem{lem}[thm]{Lemma}

\newtheorem{rmk}[thm]{Remark}
\newtheorem{cor}[thm]{Corollary}

\newtheorem{prop-defi}[thm]{Proposition-Definition}
\newtheorem{thm-defi}[thm]{Theorem-Definition}
\newtheorem{lem-defi}[thm]{Lemma-Definition}

\newtheorem{question}[thm]{Question}
\newtheorem{proposal}[thm]{Proposal}

\theoremstyle{definition}
\newtheorem{defi}[thm]{Definition}
\newtheorem{assum}[thm]{Assumption}

\newcommand{\aA}{\mathcal{A}}
\newcommand{\bB}{\mathcal{B}}
\newcommand{\cC}{\mathcal{C}}
\newcommand{\dD}{\mathcal{D}}

\newcommand{\fF}{\mathcal{F}}

\newcommand{\hH}{\mathcal{H}}

\newcommand{\lL}{\mathcal{L}}

\newcommand{\oO}{\mathcal{O}}
\newcommand{\pP}{\mathcal{P}}
\newcommand{\qQ}{\mathcal{Q}}
\newcommand{\rR}{\mathcal{R}}

\newcommand{\tT}{\mathcal{T}}
\newcommand{\uU}{\mathcal{U}}
\newcommand{\vV}{\mathcal{V}}

\newcommand{\zZ}{\mathcal{Z}}

\newcommand{\bbC}{\mathbb{C}}

\newcommand{\bbH}{\mathbb{H}}

\newcommand{\bbL}{\mathbb{L}}

\newcommand{\bbP}{\mathbb{P}}
\newcommand{\bbQ}{\mathbb{Q}}
\newcommand{\bbR}{\mathbb{R}}

\newcommand{\bbZ}{\mathbb{Z}}

\newcommand{\Hom}{\mathop{\rm Hom}\nolimits}

\newcommand{\Cone}{\mathop{\rm Cone}\nolimits}
\newcommand{\ch}{\mathop{\rm ch}\nolimits}

\newcommand{\Spec}{\mathop{\rm Spec}\nolimits}

\newcommand{\Coh}{\mathop{\rm Coh}\nolimits}
\newcommand{\Per}{\mathop{\rm Per}\nolimits}

\newcommand{\Imm}{\mathop{\rm Im}\nolimits}

\newcommand{\Ker}{\mathop{\rm Ker}\nolimits}

\newcommand{\Stab}{\mathop{\rm Stab}\nolimits}

\newcommand{\Amp}{\mathop{\rm Amp}\nolimits}
\newcommand{\NS}{\mathop{\rm NS}\nolimits}
\newcommand{\Ei}{\mathop{\rm Ei}\nolimits}
\newcommand{\Vect}{\mathop{\rm Vect}\nolimits}

\newcommand{\excl}[1]{\langle{#1}\rangle_{\text{ex}}}

\begin{document}

\title{The noncommutative MMP for blowup surfaces}
\author{Tomohiro Karube}
\date{\today}

\address{Graduate School of Mathematical Sciences, The university of Tokyo, Meguro-ku, Tokyo, 153-8914, Japan.}
\email{karube-tomohiro803@g.ecc.u-tokyo.ac.jp}

\begin{abstract}
  We study the noncommutative minimal model program for  blowups of surfaces.
  The program, as defined by Halpern-Leistner, is designed to construct a quasi-convergent path in the space of Bridgeland stability conditions.
  In this paper, we construct a family of quasi-convergent paths in the case of blowups of surfaces.
  These paths provide different semiorthogonal decompositions through parameter transformation.
\end{abstract}

\maketitle
\setcounter{tocdepth}{1}
\tableofcontents
\tableofcontents

\section{Introduction}
\subsection{Motivation}
Let $X$ be a smooth projective variety over $\bbC$.
The minimal model program (MMP) is a program for constructing a sequence of birational maps
\[
X = X_0 \dashrightarrow X_1 \dashrightarrow \cdots \dashrightarrow X_n = X_{\min}.
\]
Here, $X_{\min}$ is a minimal model of $X$ or a Mori fiber space.
Each birational map $X_i \dashrightarrow X_{i+1}$ is a flip or a divisorial contraction.
The MMP is a fundamental method in the study of birational geometry of algebraic varieties.
When $X$ is a surface, each step is a contraction of a $(-1)$ curve.
In the case of 3-folds, a flip can occur in the MMP (cf.\cite{KM:98}).

Through the MMP,  the derived category of a variety can become smaller.
In \cite{BO:95}, Bondal and Orlov have shown the relationship between semiorthogonal decompositions of derived categories and birational maps of varieties.
If $X_i$ and $X_{i+1}$ are smooth and a birational map $X_i \dashrightarrow X_{i+1}$ is a divisorial contraction or a simple flip,
then it is proved in \cite{BO:95} that the derived category $\dD^b(X_{i+1})$ is a part of a semiorthogonal decomposition of $\dD^b(X_i)$.
The above result was generalized by Kawamata in \cite{Kaw:02} to the case of orbifolds.
There are also studies pointing out the relationship between derived categories and birational geometry (\cite{Kaw:02b,Bri:00}).

The \emph{noncommutative minimal model program (NCMMP)}, proposed by Halpern-Leistner in \cite{HL:23}, is an analogue of MMP in derived categories.
The NCMMP aims to divide the derived category into smaller components through semiorthogonal decomposition
by using the space of Bridgeland stability conditions.
The space of Bridgeland stability conditions is defined in \cite{Bri:07}, and 
it has been studied as an invariant of triangulated categories.
In the NCMMP, we consider a path in the space of Bridgeland stability conditions to construct a semiorthogonal decomposition of the derived category.
Such a path is called a \emph{quasi-convergent path} defined in \cite{HL:23,HLJR}.

Semiorthogonal decompositions of a variety is also related to the quantum cohomology.
Let $F$ be a Fano variety.
The quantum cohomology is a deformation of the cohomology ring of $F$ via Gromo-Witten invariants, 
whose product is denoted by $\star_{\tau}$ for a parameter $\tau \in H^*(X)$.
The quantum cohomology ring $(H^*(F),\star_{\tau})$ defines the quantum differential equation
\[
t \frac{d \Phi_t}{dt} + c_1(F) \star_{c_1(F)\log(t)} \Phi_t= 0.
\]
Dubrovin's conjecture says that the semi-simplicity of the quantum cohomology of $F$ is equivalent to 
the existence of a full exceptional collection on $\dD^b(F)$ \cite{Dub:98}.
The Stokes matrix of the above equation plays an important role in this conjecture.
The conjecture is partially refined and is known as the Gamma Conjecture II.
The Gamma Conjecture II, proposed in \cite{GGI}, expects a correspondence between the behavior of solutions to the quantum differential equation and exceptional objects in the derived category of $F$. 
In the NCMMP, we have the following problem inspired by the Gamma Conjecture II.
\begin{question}[{\cite[Proposal III.]{HL:23}}]\label{question:ncmmp}
  Let $\pi \colon X \rightarrow Y$ be a contraction (i.e. surjective map with connected fibers) of a smooth projective variety.
  Let $\Phi_t$ be an endomorphism of $H^*(X)$ such that $\Phi_t(\alpha)$ satisfies the truncated quantum differential equation for all $\alpha \in H^*(X)$.
  \begin{enumerate}
    \item Does a family of central charges $Z_t = \int_{X} \Phi_t$ lift to a family of stability conditions on $X$?
    \item If the answer to the above question is yes, then is the family of stability conditions a quasi-convergent path?
  \end{enumerate}
\end{question}
In \cite{HL:23}, only the fundamental solution of the quantum differential equation is considered.
However, we consider all paths satisfying the truncated quantum differential equation in this paper.

\subsection{Main results}
In this paper, we study the NCMMP for the blowup of a surface $\pi \colon X \rightarrow Y$.
There is a semiorthogonal decomposition $\langle \oO_C(-1),\dD^b(Y)  \rangle$ of $\dD^b(X)$, where $C$ is the exceptional divisor of $\pi$.
The semiorthogonal decompositions mutation-equivalent to $\langle \oO_C(-1),\dD^b(Y)  \rangle$ are given by 
\[
\langle \oO_C(k),f_{k+1}^L\dD^b(Y)  \rangle \text{ and } \langle f^L_{k}\dD^b(Y) ,\oO_C(k) \rangle \text{ for } k \in \bbZ.
\]
Here, $f_{k}^L (-) = f^*(-) \otimes \oO_C(-k)$.
Let $\Stab(X)$ be the space of Bridgeland stability conditions on $X$.
We mainly consider the normalized stability conditions whose central charges $Z$ are normalized to be $Z(\oO_x) = -1$.
Let $\Stab_n(X)$ be the space of normalized stability conditions on $X$.
The space $\Stab_n(X)$ contains the following open subset $U(X)$:
\[
U(X) = \left\{ \sigma \in \Stab_n(X) \colon \text{all skyscraper sheaves } \oO_x \text{ are stable of phase one} \right\}.
\]
We call $U(X)$ the normalized geometric chamber of $X$.
Other regions of $\Stab_n(X)$ called the gluing region are defined by semiorthogonal decompositions of $\dD^b(X)$.
We discuss two type of gluing regions $U(\rR_{0})$ and $U(\lL_{-1})$, see \Cref{def:region_prime} and \Cref{defi:region} in details.
These gluing regions are defined by the notion of glued stability conditions induced in \cite{CP:09}.

We construct two families of quasi-convergent paths in $\Stab(X)$.
One of the above families can be taken so that  its initial values lies in $U(X)$.
The other family induces different semiorthogonal decompositions through parameter transformation.
We consider the following central charges on $X$ defined to be
\begin{equation}
  \zZ(\lambda,t)(E) = g(\lambda,t)\ch_1(E)C + Z_{0}(E) 
\end{equation}
where $t \in \bbR_{\geq T_0}$ for some $T_0 >1$.
The function $g(\lambda,t)$ satisfies the truncated quantum differential equation for a blowup $\pi \colon X \rightarrow Y$.
Here, $\lambda$ is a parameter in $\bbR \oplus i[-1,1]$.
The truncated quantum differential equation of $\pi$  is defined to be 
\begin{equation}\label{eq:qdeINTRO}
  \frac{d}{dt}\left(t \frac{d g(\lambda, t)}{dt}\right) = \lambda t \frac{d g(\lambda, t)}{dt}.
\end{equation}
The following is the main theorem of this paper.
\begin{thm}[\Cref{thm:startgeom}]
  Let $Y$ be a smooth projective surface and $\pi \colon X \rightarrow Y$ be a blowup of a point.
  Let $T_0 \in \bbR_{\geq 1}$ and $\sigma_{T_0}$ be a stability condition in the geometric chamber $U(X)$.
  Fix a semiorthogonal decomposition which is mutation-equivalent to $\langle \oO_C(-1),\dD^b(Y)  \rangle$.
  Then, there exist a parameter $\lambda$ and  a solution $g(\lambda,t)$ of the truncated quantum differential equation (\ref{eq:qdeINTRO})
  such that the central charge $\zZ(\lambda,t)$ defines a quasi-convergent path $\sigma_{\bullet}$.
  Moreover, the path $\sigma_{\bullet}$ induces the fixed semiorthogonal decomposition of $\dD^b(X)$.
\end{thm}
Note that the function $g(\lambda, t)$ appearing in the main theorem may not be unique.
However, the function $g(\lambda, t)$ constructed in the proof is discontinuous at $\lambda = 0$.
The next theorem asserts the function $g(\lambda, t)$ can be taken to be a continuous function when the initial value of the path lies in a gluing region $U(\rR_{0})$.
It was discovered that two semiorthogonal decompositions $\langle \oO_C(-1),\dD^b(Y)  \rangle$ and $\langle \dD^b(Y) ,\oO_C \rangle$
 can be obtained through continuous deformations of a quasi-convergent path according to this theorem.

\begin{thm}[\Cref{thm:startgluing}]
  Let $Y$ be a smooth projective surface and $\pi \colon X \rightarrow Y$ be a blowup of a point.
  Fix a compact subset $K \subset \bbR \oplus i[-1,1]$.
  Then, there exists a solution $g(\lambda,t)$ of the truncated quantum differential equation such that the central charge $\zZ(\lambda,t)$ lifts to a path $\sigma_t \in \Stab(X)$.
  The function $g(\lambda,t)$ is continuous on $(\lambda,t) \in K\times \bbR_{\geq T_0}$
  Moreover, if $\Imm \lambda \not= 0$, then the path $\sigma_{\bullet}$ is quasi-convergent. In the case of $\Imm \lambda < 0$, the path
induces a semiorthogonal decomposition $\langle \oO_C(-1),\dD^b(Y)  \rangle$. In the case of $\Imm \lambda > 0$, the path induces
a semiorthogonal decomposition $\langle \dD^b(Y) ,\oO_C \rangle$.
\end{thm}

To show the above theorems, we study the walls of the gluing region of the space of Bridgeland stability conditions.
We consider gluing regions $U(\rR_{0})$ and $U(\lL_{-1})$ corresponding to semiorthogonal decompositions $\langle \oO_C(-1),\dD^b(Y)  \rangle$ and $\langle \dD^b(Y) ,\oO_C \rangle$.

\begin{thm}[\Cref{thm:gluingregion}]
  Let $\sigma_Y = (Z_Y,\aA_Y)$ be a normalized geometric stability condition on $Y$, and let $U(\rR_{0})_{\sigma_Y}$ be a set defined by 
\[
  U(\rR_{0})_{\sigma_Y} = \left\{ \sigma \in U(\rR_{0}) \colon \sigma \text{ is glued from }\tau_{\lambda} \text{ and } \sigma_Y.\right\}.
\]
  The boundary of $U(\rR_{0})_{\sigma_Y}$ consists of disjoint union of two components $W_0$ and $W_1$, which are defined by
\begin{align*}
  &W_0 = \{\tau \in \partial U(\rR_{0}) \colon \oO_C(-1)[1] \text{ is semistable of phase } 1 \},\\
  &W_{-1} = \{\tau \in \partial U(\rR_{0}) \colon \oO_C(-1)[1] \text{ is semistable of phase } 0 \}.
 \end{align*}
 Moreover, the following two properties hold:
 \begin{enumerate}
  \item For any $\tau \in W_0$, if $\oO_C$ and $\oO_C(-1)[1]$ are both $\tau$-stable, then $\tau$ lies in the boundary of the normalized geometric chamber $U(X)$.
  \item The set $W_{-1}$ is contained in $U(\lL_{-1})$.
 \end{enumerate}
\end{thm}

\subsection{Relation to existing work}
In \cite{HL:23}, Halpern-Leistner proposed the NCMMP, and showed Question \ref{question:ncmmp} for the projective line.
The NCMMP for $\bbP^2$ was studied in \cite{Zu:24}.
In \cite{HLJR}, it has been shown that there exists a quasi-convergent path for any semiorthogonal decomposition.
The constructed path lies in the gluing region corresponding to the semiorthogonal decomposition.
There is not much research on the NCMMP and quasi-convergent paths starting from other chambers.

However, there are several works that shows the relationship between the Bridgeland stability conditions and birational geometry.
There is a relationship between birational morphism of moduli spaces on a variety $X$ and walls in $\Stab(X)$.
Let $M$ be a moduli space of stable sheaves on a K3 surface or $\bbP^2$.
It is known that the minimal model of $M$ can be obtained by changing the stability condition on the corresponding surface \cite{ABCH,BM:14,BMW}.
By considering the moduli space of stable objects  whose numerical class is the same as the skyscraper sheaf,
Toda showed that the MMP of surfaces can be determined by the structure of the derived category alone \cite{Tod:14}.

\subsection{Acknowledgement}
The author is greatful to my advisor Yukinobu Toda for valuable comments.
I would also like to thank W. Hara and T. Yoshida for useful conversations,
 and Nantao Zhang for pointing out the mistakes in the previous version.
This work was supported by JSPS KAKENHI Grant Number 	24KJ0713.
The author would like to thank the Hausdorff Research Institute for
Mathematics funded by the Deutsche Forschungsgemeinschaft (DFG, German Research
Foundation) under Germany's Excellence Strategy -EXC-2047/1 - 390685813 
for their hospitality during the preparation of this work.
\subsection{Notation and  Convention}
In this paper, all the varieties are defined over $\bbC$.
For a variety $X$, $\dD^b(X)$ denotes the bounded derived category of coherent sheaves on $X$.
For a proper morphism $f \colon X \rightarrow Y$ of varieties, we consider the derived pullback $\bbL f^*$ and the derived pushforword $\bbR f_*$.
We will write $\bbL f^*$ and $\bbR f_*$ simply $f^*$ and $f_*$ 
when no confusion can arise.
We denote the cohomology functor by $\hH^i_{\aA}(E)$ with respect to a heart $\aA$.
For an object $E\in \dD^b(X)$, we write it $\hH^i(E)$ for short when no confusion can arise.
Let $H^i(X,E)$ be the $i$-th cohomology  group of an object $E \in\dD^b(X) $.

\section{Preliminaries}

\subsection{Recollements}
In this section, we define recollements of triangulated categories.
The notion of recollements was introduced by Beilinson, Bernstein, and Deligne in \cite{BBD:82}.
They showed that t-structures can be glued with respect to the recollements.
We will review the standard properties of recollements and their relation to semiorthogonal decompositions.
For the proofs, we refer the reader to \cite{Jor:10,LV:12,LVY:14}.

\subsubsection{Recollements and semiorthogonal decompositions}
\begin{defi}\label{defi:rec}
Let $\aA,\bB,\cC$ be triangulated categories.
A recollement of $\bB$ between $\aA$ and $\cC$ consists of the following diagram of functors:
\[
  \rR \colon
  \begin{tikzcd}
  \aA \arrow[r,"i" description] & \bB \arrow[l,bend right=40,"i_L" description] \arrow[l,bend left= 40,"i_R"] \arrow[r,"j_R" description]
   & \cC\arrow[l,bend right=40,"j" description ] \arrow[l,bend left= 40,"j_{RR}" description],
    \end{tikzcd}
\]
where the functors $i,i_L,i_R,j,j_R,j_{RR}$ satisfy the following conditions:
\begin{enumerate}
  \item $i$ and $j$ are fully faithful functors.
  \item $j_R \circ i =0$.
  \item $i_L \dashv i \dashv i_R$ and $j \dashv j_R \dashv j_{RR}$.
  \item For any object $b$ in $\bB$, there exist two distinguished triangles:
  \begin{align*}
    j j_R b \rightarrow b \rightarrow i i_Lb \overset{+}{\rightarrow} ,\\
    i i_R b \rightarrow b \rightarrow j_{RR}j_R b \overset{+}{\rightarrow} .
  \end{align*}
\end{enumerate}
\end{defi}

\begin{lem-defi}[{\cite[Lemma 2.1.]{LV:12}}]\label{lem-defi:actiononauto}
Let $\rR$ be a recollement of $\bB$ between $\aA$ and $\cC$:
\[
  \rR \colon
  \begin{tikzcd}
  \aA \arrow[r,"i" description] & \bB \arrow[l,bend right=40,"i_L" description] \arrow[l,bend left= 40,"i_R"] \arrow[r,"j_R" description]
   & \cC\arrow[l,bend right=40,"j" description ] \arrow[l,bend left= 40,"j_{RR}" description].
    \end{tikzcd}
\]
For any autoequivalence $\Phi$ of $\bB$, we can define the recollement $\Phi \rR$:
\[
  \Phi \rR \colon
  \begin{tikzcd}
  \aA \arrow[r,"\Phi \circ i" description] & \bB \arrow[l,bend right=70,"i_L \circ \Psi" description] \arrow[l,bend left= 70,"i_R \circ \Psi"] \arrow[r,"j_R \circ \Psi"]
   & \cC\arrow[l,bend right=70,"\Phi \circ j" description ] \arrow[l,bend left= 70,"\Phi \circ j_{RR}" description].
    \end{tikzcd}
\]
Here, $\Psi$ is a quasi-inverse of $\Phi$.
\end{lem-defi}

\begin{lem-defi}[{\cite[Theorem 2.1.]{LVY:14},\cite[Theorem 5.]{BK:89}}]
  Let $\rR$ be a recollement of $\bB$ between $\aA$ and $\cC$:
  \[
  \rR \colon
  \begin{tikzcd}
  \aA \arrow[r,"i" description] & \bB \arrow[l,bend right=40,"i_L" description] \arrow[l,bend left= 40,"i_R"] \arrow[r,"j_R" description]
   & \cC\arrow[l,bend right=40,"j" description ] \arrow[l,bend left= 40,"j_{RR}" description].
    \end{tikzcd}
\]
Assume that $\bB$ has the Serre functor $S_{\bB}$.
Then, there exist upper and lower recollements $\rR_U$ and $\rR_L$.
\[
  \rR_U \colon
  \begin{tikzcd}
    \cC \arrow[r,"j" description] & \bB  \arrow[l,bend left= 70,"j_R" description] \arrow[l,bend right = 70, "j_L"description] \arrow[r,"i_L" description] 
    & \aA \arrow[l,bend left=70,"i" description] \arrow[l,bend right=70,"i_{LL}" description], \ \ 
    \end{tikzcd}
  \rR_L \colon
  \begin{tikzcd}
    \cC \arrow[r,"j_{RR}" description] & \bB  \arrow[l,bend left= 70,"j_{RRR}" description] \arrow[l,bend right = 70, "j_R"description] \arrow[r,"i_R" description] 
    & \aA \arrow[l,bend left=70,"i_{RR}" description] \arrow[l,bend right=70,"i" description].
    \end{tikzcd}
\]

Here, $j_L$ is the left adjoint of $j$ and $i_{LL}$ is the left adjoint of $i_L$ defined to be
\[
j_L  = S_{\cC}^{-1} \circ j_{R} \circ S_{\bB}, \ \ i_{LL} = S_{\bB}^{-1} \circ i \circ S_{\aA},
\]
where $S_{\aA}$ and $S_{\cC}$ are Serre functors of $\aA$ and $\cC$ respectively.
The functors $j_{RRR}$ and $i_{RR}$ are defined similarly.
\end{lem-defi}

Next, we consider the relation between recollements and semiorthogonal decompositions which have only two components.
We recall the definition of a semiorthogonal decomposition.
\begin{defi} \label{defi:sod}
  Let $\dD$ be a triangulated category.
  A semiorthogonal decomposition of $\dD$ is a sequence of full triangulated subcategories $\dD_1,\dots,\dD_n$ of $\dD$ such that the following conditions are satisfied:
  \begin{enumerate}
    \item For any $n,m$, $\Hom_{\dD}(\dD_n,\dD_m) = 0$ if $n > m$.
    \item For any object $E$ in $\dD$, there exists a filtration 
    \[
      0=E_{n} \rightarrow \cdots \rightarrow E_1 \rightarrow E_0 = E
      \]such that each cone $\mathrm{Cone}(E_i \rightarrow E_{i-1})$ lies in $\dD_i$.
  \end{enumerate}
\end{defi}
We write $\langle (\dD_1,\iota_1),\dots,(\dD_n,\iota_n) \rangle$ for a semiorthogonal decomposition of $\dD$.
Here, $\iota_m$ is the inclusion functor from $\dD_m$ to $\dD$.
For simplicity, we write $\langle \dD_1,\dots \dD_n \rangle$ for a semiorthogonal decomposition.
\begin{defi}
  Let $\aA$ be an full subcategory of a triangulated category $\dD$.
  Here, if $\aA$ is called an admissible subcategory then $\aA$ is a full triangulated subcategory of $\dD$ and
   the inclusion functor $i \colon \aA \rightarrow \dD$ has left and right adjoints.
\end{defi}

\begin{rmk}
  We mainly consider triangulated categories with Serre functors (cf.\cite[Definition 1.28.]{Huy:06} ).
  If $\bB$ has the Serre functor and we have a recollement of $\bB$ between $\aA$ and $\cC$,
   then $\aA$ and $\cC$ have  Serre functors and they are admissible subcategories of $\bB$.
   The detailed proof appears in \cite{BK:89,Jor:10}
\end{rmk}

We focus on the semiorthogonal decompositions with two components $\langle  D_1, D_2\rangle $.
\begin{prop}[{\cite{BK:89}}]
  Let $\dD$ be a triangulated category with the Serre functor $S$.
  We assume that $\dD$ has a semiorthogonal decomposition $\langle (\dD_1,i_1),(\dD_2,i_2) \rangle$, and $\dD_1$ and $\dD_2$ are admissible.
  We denote by $S_1$ and $S_2$ the Serre functors of $\dD_1$ and $\dD_2$ respectively.
  Then, $\langle(D_2,i_2),(D_1,S^{-1}\circ i_1 \circ S_1) \rangle$ and $\langle(D_2,S \circ i_2 \circ S_2^{-1}),(D_1,i_1) \rangle$ are semiorthogonal decompositions of $\dD$.
\end{prop}
\begin{proof}
  We only prove that $\langle(D_2,i_2),(D_1,S^{-1}\circ i_1 \circ S_1) \rangle$ is a semiorthogonal decomposition of $\dD$.
  The proof of the other case is similar.
  The first condition is clear.
  Indeed, we have 
  \begin{align*}
    \Hom_{\dD}(S^{-1}\circ i_1 \circ S_1 (D_1),i_2(D_2)) &\cong \Hom_{\dD}(i_2(D_2),i_1\circ S_1 (D_1))^* \\
    &= \Hom_{\dD}(i_2(D_2),i_1 (D_1)) = 0.
  \end{align*}
  For the second condition, we consider an object $E$ in $\dD$.
  Set  $i_1^{\prime} = S^{-1}\circ i_1 \circ S_1$.
  Let $\mu_n$  be the left adjoint of $i_n$ for $n=1,2$.
  Since $i_1^{\prime}$ is the left adjoint, there is a distinguished triangle
  \[
    i_1^{\prime} \mu_1 E \rightarrow E \rightarrow C \overset{+}{\rightarrow},
  \] 
  where $C$ is a cone of $i_1^{\prime} \mu_1 E \rightarrow E$.
  We need to show $C$ lies in $\dD_2$.
  Since $\langle \dD_1,\dD_2 \rangle$ is a semiorthogonal decomposition,
  the subcategory $\dD_2 \subset \dD $ is the same as 
  \[
  \left\{ d \in \dD \colon \Hom(d,\dD_1) = 0 \right\}.
  \]
  By adjunctions, we have $ \Hom(i_1^{\prime} \mu_1 E, i_1 D_1) \cong \Hom(E, i_1 D_1) $.
  Then, since $\Hom(C,i_1(D_1))$ vanishes, $C$ lies in $\dD_2$.
\end{proof}
For a semiorthogonal decomposition $\langle \dD_1 ,\dD_2 \rangle$, 
 $\langle(\dD_2,S  i_2 S_2^{-1}),(\dD_1,i_1) \rangle$ and 
 $\langle(\dD_2,i_2),(\dD_1,S^{-1} i_1 S_1) \rangle$ are called the \textit{left} and \textit{right} mutation respectively.

\begin{thm}[{\cite[Section 1]{BB:18}}]
  Let $\dD$ be a triangulated category with the Serre functor.
  Consider the following two sets:
  \begin{align*}
   \mathrm{Rec}(\dD) &= \left\{ \rR \colon \rR \text{ is a recollement of } \dD \right\},\\
    \mathrm{SOD}_{2}(\dD) &= \left\{ \left\{ \dD_1,\dD_2 \right\} \colon
      \begin{array}{l}
         \text{admissible subcategories } \dD_1,\dD_2  \\
        \text{such that $\langle \dD_1,\dD_2 \rangle$ is a semiorthogonal decomposition of } \dD
        \end{array} 
    \right\}.\\
  \end{align*}

  Then, there is a bijection $\Phi$ between $\mathrm{Rec}(\dD)$ and $\mathrm{SOD}_2(\dD)$.
  The bijection is given by the following:
  \[
    \Phi: 
    \begin{tikzcd}
      \aA \arrow[r,"i" description] & \dD \arrow[l,bend right=40,"i_L" description] \arrow[l,bend left= 40,"i_R"] \arrow[r,"j_R" description]
       & \cC\arrow[l,bend right=40,"j" description ] \arrow[l,bend left= 40,"j_{RR}" description]
    \end{tikzcd}
    \mapsto \left\{ (\aA,i),(\cC,j) \right\}.
\]
The bijection sends the upper recollement of $\rR$ to the right mutation of $\Phi(\rR)$.
Similarly, the lower recollement of $\rR$ is sent to the left mutation of $\Phi(\rR)$.
\end{thm}
\begin{proof}
  First, we show that $\Phi(\rR)$ is a semiorthogonal decomposition.
  Let $\rR$ be a recollement of $\dD$:
  \[
  \begin{tikzcd}
    \aA \arrow[r,"i" description] & \dD \arrow[l,bend right=40,"i_L" description] \arrow[l,bend left= 40,"i_R"] \arrow[r,"j_R" description]
     & \cC\arrow[l,bend right=40,"j" description ] \arrow[l,bend left= 40,"j_{RR}" description].
  \end{tikzcd}
  \]
  Then, subcategories $i\colon\aA \hookrightarrow \dD$ and $j \colon \cC \hookrightarrow \dD$  are admissible.
  We need to show that $\langle(\aA,i),(\cC,j) \rangle$ is a semiorthogonal decomposition of $\dD$.
  Since $j_R i = 0$, we have $\Hom(j \cC,i \aA) =0$.
  The condition (2) in Definition \ref{defi:sod} follows from the first distinguished triangle in the condition (4) of Definition \ref{defi:rec} .

  Next, we construct the inverse map $\Psi \colon \mathrm{SOD}_2(\dD) \rightarrow \mathrm{Rec}(\dD)$.
  Let $\langle (\dD_1,i_1),(\dD_2,i_2) \rangle$ be a semiorthogonal decomposition of $\dD$.
  We define the recollement $\Psi( \left\{ (\dD_1,i_1),(\dD_2,i_2) \right\})$ as follows:
  \[
  \Psi(\left\{ (\dD_1,i_1),(\dD_2,i_2) \right\}) \coloneqq
  \begin{tikzcd}
    \dD_1 \arrow[r,"i_1" description] & \dD \arrow[l,bend right=40,"\mu_1" description] \arrow[l,bend left= 40,"\lambda_1" description] \arrow[r,"\lambda_2" description]
     & \dD_2\arrow[l,bend right=40,"i_2" description ] \arrow[l,bend left= 40,"\lambda_{2,R}"description].
  \end{tikzcd}
  \]
  Since $\dD_1$ is admissible, there are the adjunctions $\mu_1 \dashv i_1 \dashv \lambda_1$.
  The adjunctions $i_2 \dashv \lambda_2 \dashv \lambda_{2,R}$ are defined to be
  \[
    \lambda_{2,R} \coloneqq S \circ i_2 \circ S_2^{-1}.
  \]
  Here, $S$ and $S_2$ are the Serre functors of $\dD$ and $\dD_2$ respectively.
  Set $\rR = \Psi(\left\{ (\dD_1,i_1),(\dD_2,i_2)\right\}) $.
  We need to show that $\rR$ is a recollement of $\dD$.
  The conditions (1-3) in Definition \ref{defi:rec} are clear.
  We only need to show the condition (4).
  By the definition of semiorthogonal decompositions, there is  a distinguished triangle
  \[
  i_2 \lambda_{2} E \rightarrow E \rightarrow i_1 \mu_1 E \overset{+}{\rightarrow} 
  \] for any object $E$ in $\dD$.
  Therefore, we need to show that there exists the following distinguished triangle in Definition \ref{defi:rec}:
  \[
    i_1 \lambda_1 E \rightarrow E \rightarrow \lambda_{2,R} \lambda_{2} E \overset{+}{\rightarrow}.
  \]
  Consider the right mutation $\langle(\dD_2,S  i_2 S_2^{-1}),(\dD_1,i_1) \rangle$ of $\langle D_1,D_2 \rangle$.
  We have the following distinguished triangle for any $E \in \dD$:
  \[
    i_1 \lambda_1 E \rightarrow E \rightarrow S i_2 S_2^{-1} i_2 \lambda_2 E \overset{+}{\rightarrow}.
  \]
  Since $S i_2 S_2^{-1} i_2 \cong \lambda_{2,R} \lambda_2$, we have the desired distinguished triangle.
  The compatibility of the upper and lower recollements with the semiorthogonal decompositions is clear.
\end{proof}

\subsubsection{BBD-induction}
Next, we discuss the BBD-induction for t-structures.
\begin{defi}
  Let $\dD$ be a triangulated category.
  A t-structure on $\dD$ is a pair of full subcategories $(\dD^{\leq 0},\dD^{\geq 0})$ of $\dD$ such that the following conditions are satisfied:
  \begin{enumerate}
    \item $\Hom(\dD^{\leq 0},\dD^{\geq 0}[-1]) = 0$.
    \item $\dD^{\leq 0}[1] \subset \dD^{\leq 0}$ and $\dD^{\geq 0}[-1] \subset \dD^{\geq 0}$.
    \item For any object $E$ in $\dD$, there exists a distinguished triangle
    \[
      E^{\leq 0} \rightarrow E \rightarrow E^{\geq 0}[-1] \overset{+}{\rightarrow}
    \] such that $E^{\leq 0}$ lies in $\dD^{\leq 0}$ and $E^{\geq 0}$ lies in $\dD^{\geq 0}$.
  \end{enumerate}
\end{defi}

\begin{defi}
  Let  $(\dD^{\leq 0},\dD^{\geq 0})$ be a t-structure on a triangulated category $\dD$.
  A t-structure $(\dD^{\leq 0},\dD^{\geq 0})$  is called bounded if 
  \[
  \dD = \bigcup_{n \in \bbZ} \dD^{\leq 0}[-n] = \bigcup_{n \in \bbZ} \dD^{\geq 0}[-n].
  \]
\end{defi}

The heart of a t-structure $(\dD^{\leq 0},\dD^{\geq 0})$  is defined to be the full subcategory
\[
 \aA = \dD^{\leq 0} \cap \dD^{\geq 0}.
\]
It is known that the heart $\aA$ is an abelian category.
The BBD-induction is a method to construct a t-structure on a recollement.
\begin{thm}{\cite[Theorem 1.4.10.]{BBD:82}}\label{thm:BBD}
  Let $\rR$ be a recollement of a triangulated category $\dD$:
  \[
  \rR \colon
  \begin{tikzcd}
  \cC_1 \arrow[r,"i" description] & \dD \arrow[l,bend right=40,"i_L" description] \arrow[l,bend left= 40,"i_R"] \arrow[r,"j_R" description]
   & \cC_2 \arrow[l,bend right=40,"j" description ] \arrow[l,bend left= 40,"j_{RR}" description].
    \end{tikzcd}
  \]
  Assume that  $(\cC_1^{\leq 0},\cC_1^{\geq 0})$ and $(\cC_2^{\leq 0},\cC_2^{\geq 0})$ are bounded t-structures.
  Then, there exists a unique bounded t-structure on $\dD$ defined to be
\begin{align*}
  \dD^{\leq 0} &= \left\{ E \in \dD \colon j_R E \in \cC_2^{\leq 0 } , i_L E \in \cC_1^{\leq 0}  \right\},\\
  \dD^{\geq 0} &= \left\{ E \in \dD \colon  j_R E \in \cC_2^{\geq 0 } , i_R E \in \cC_1^{\geq 0}  \right\}.
\end{align*}
  \end{thm}

\begin{rmk}
  The heart of the t-structure constructed by the BBD-induction is given by
  \[ \aA = 
    \left\{ E \in \dD  \colon
    \begin{array}{l}
    \Hom(E, i(\aA_1)[n]) = 0 \text{ for $n<0$},\\
    \Hom(i(\aA_1)[n], E) = 0 \text{ for $n>0$}.
    \end{array}
    \right\}
  \]
  Here, $\aA_1$ and $\aA_2$ are the hearts of the t-structures on $\cC_1$ and $\cC_2$ respectively.
\end{rmk}

\subsection{Stability conditions}\label{pre:stab}
We review the notion of Bridgeland stability conditions. 
We also discuss the stability conditions on surfaces.

\subsubsection{Stability conditions and Bridgeland's deformation theorem}
First, we define stability functions on abelian categories.
Let $\aA$ be an abelian category and $K(\aA)$ be its Grothendieck group.
\begin{defi}
  \begin{enumerate}
  \item 
  A stability function on $\aA$ is a group homomorphism $Z : K(\aA) \rightarrow \bbC$ such that for all $0\not=E \in \aA$
  \[
  Z(E) \in \bbH \cup \bbR_{<0}.
  \]
  Here, $\bbH$ is the upper half plane.
  
  \item
  Let $Z$ be a stability function on $\aA$. The phase of an object $0 \not = E \in \aA$ is defined to be \[\phi(E) = (1/\pi) \arg Z(E) \in \left( 0,1\right].\]
  
  \item
  An object $0 \not = E \in \aA$ is said to be semistable (resp. stable) if every subobject $0 \not = A \subset E$ satisfies $\phi(A) \leq(resp. <) \phi(E)$.
  \item We say that $Z$ satisfies the Harder-Narasimhan property (HN property) if for every object $E\in\aA$ there is a filtration
  \[
  0 = E_0 \subset E_1 \subset \cdots \subset E_{n-1} \subset E_n = E
  \]
  whose factors $F_j = E_j / E_{j-1}$ are semistable objects of $\aA$ with
  \[
  \phi(F_1) >\phi(F_2) > \cdots> \phi(F_n).
  \]
  \end{enumerate}
  \end{defi}
  Let $\dD$ be a triangulated category. 
  The notion of support property is introduced in \cite{KS:08}.
  We fix a finite free abelian group $\Lambda$ with a group homomorphism $v:K(\dD) \rightarrow \Lambda$. 
  We note that if $\aA$ is the heart of a bounded t-structure on $\dD$, then $K(\aA)$ is isomorphic to
   $K(\dD)$.
  
  \begin{defi}
  \begin{enumerate}
  \item 
  A pre-stability condition $\sigma = (Z, \aA)$ on $\dD$ with respect to $\Lambda$ consists of the heart of a bounded t-structure $\aA$ and a stability function $Z$ on $\aA$ which factors through $v:K(\aA) \cong K(\dD) \rightarrow \Lambda$ and satisfies the HN-property. 
  \item
  Let $\lVert-\rVert$ be a norm on $\Lambda_{\bbR} = \Lambda\otimes \bbR$. A pre-stability condition $\sigma = (Z, \aA)$ on $\dD$ with respect to $\Lambda$ satisfies the support property if the following holds:
  \[
  \sup \left\{  \frac{\lVert v(E)\rVert}{\lvert Z(E) \rvert}\colon E \rm{\ is\ semistable\ in\ }\aA \right\} <\infty.
  \]
  \item 
  A stability condition is a pre-stability condition satisfying the support property.
  
  \end{enumerate}
  \end{defi}
  
  \begin{lem}[{\cite[Lemma A.4.]{BMS:16}}]
  A pre-stability condition $\sigma = (Z,\aA)$ with respect to $\Lambda$ satisfies the support property if and only if there exists a quadratic form $Q$ on the vector space $\Lambda_{\bbR}$ such that
  \begin{itemize}
  \item for any semistable objects $E \in \aA$, $Q(v(E)) \geq 0$, and 
  \item $Q$ is negative definite on $\Ker Z \subset \Lambda_{\bbR}$.
  \end{itemize}
  \end{lem}
  
  \begin{defi}
  A slicing $\pP =\{\pP(\phi)\}_{\phi \in \bbR}$ of $\dD$ consists of full additive subcategories $\pP(\phi) \subset \dD$ for each $\phi$, satisfying:
  \begin{enumerate}
  \item For all $\phi \in \bbR$, $\pP(\phi+1) = \pP(\phi)[1]$.
  \item If $\phi_1 > \phi_2$ and $E_j \in \pP(\phi_j)$, then $\Hom(E_1,E_2) = 0$.
  \item For every $E \in \dD$ there exists a finite sequence 
  \[
  0 = E_0 \rightarrow E_1 \rightarrow \cdots E_m =E
  \]
  such that  $\mathrm{Cone}(E_i \rightarrow E_{i+1})$ lies in $\pP(\phi_i)$ with $\phi_0 >\phi_1 >\cdots > \phi_{m-1}$.
   
  \end{enumerate}
  
  \end{defi}

  Let $\sigma =(Z,\aA)$ be a stability condition on $\dD$. For each real number $\phi \in (0,1]$, full additive category $\pP(\phi)$ is defined as follows:
  \[
  \pP(\phi) \coloneq \{E \in \aA \colon E \text{ is semistable of the phase } \phi  \}.
  \]
  For each real number $\phi \in \bbR$,
  \[
  \pP(\phi) = \pP (\phi-k ) [k]
  \]where $k = \lfloor \phi \rfloor$. 
   It follows from the support property that each category $\pP(\phi)$ is of finite length, so each semistable object has a Jordan-Holder filtration.
  For a non-zero object $E$, we write $\phi_{\sigma}^{+} (E) = \phi_0$ and $\phi_{\sigma} ^{-}(E) = \phi_{m-1}$. The mass of $E$ is defined by $m(E) = \sum |Z(A_i)|$.
  Then, the set $\Stab_{\Lambda}(\dD)$ of stability conditions with respect to $\Lambda$ on $\dD$ has a natural topology induced by the following metric
  \[
  d(\sigma_1,\sigma_2) = \sup_{0\not = E \in \dD} \left\{ \lvert\phi_{\sigma_1}^{-}(E) - \phi_{\sigma_2}^{-}(E)\rvert,\lvert\phi_{\sigma_1}^{+}(E) - \phi_{\sigma_2}^{+}(E)\rvert,\lvert\log \frac{m_{\sigma_1}(E)}{m_{\sigma_2}(E)}\rvert \right\}.
  \]
  The topology makes the forgetful map continuous
  \[
  \begin{array}{ccc}
  \pi:\Stab(\dD)&\rightarrow& \Hom(\Lambda,\bbC).\\
  (Z,\aA)&\mapsto&Z
  \end{array}
  \]
  It follows from this result called the deformation property of stability conditions that the forgetful map $\pi$ is a local homeomorphism.
  
  \begin{thm}[{\cite[Theorem 1.2.]{Bri:07},\cite[Lemma A.5.]{BMS:16}},{\cite[Theorem 1.2.]{Bay:19}}]\label{thm:deformation}
  Let $Q$ be a quadratic form on $\Lambda \otimes\bbR$ and assume that the stability condition $\sigma = (Z ,\pP)$ satisfies the support property with respect to $Q$. 
  Consider the open set of $\Hom(\Lambda,\bbC)$ consisting of central charges whose kernel is negative-definite with respect to $Q$, and let $U$ be the connected component containing $Z$.
  Then:
  \begin{enumerate}
  \item There is an open neighborhood $\sigma \in U_{\sigma} \subset \Stab_{\Lambda}(\dD)$ such that the restriction $\zZ : U_{\sigma} \rightarrow U$ is a covering map.
  \item All stability conditions in $U_{\sigma}$ satisfy the support property with respect to $Q$.
  \end{enumerate}
  \end{thm}

\subsubsection{Stability conditions on surfaces}
Next, we introduce stability conditions on surfaces constructed in \cite{Bri:08,ABL:13,MS:17,Del:23}.
Let $X$ be a smooth projective surface,
and $K(X)$ be the Grothendieck group of coherent sheaves on $X$.
We denote by $\Lambda_X$ the image of the Chern character map $\ch \colon K(X) \rightarrow H^*(X;\bbZ)$.
In this paper, we consider stability conditions with respect to $\Lambda_X$.
  We denote by $\NS(X)$ the N\'eron-Severi group of $X$.
  Set ${N}^1(X) = \NS(X) \otimes \bbR$.
 To construct t-structures on $\dD^b(X)$, we use torsion pairs.
 \begin{defi}
  Let $\aA$ be an abelian category.
  A torsion pair $(\tT,\fF)$ in an abelian category $\aA$ is a pair of full subcategories $\tT$ and $\fF$ 
  of $\aA$ with the following conditions:
  \begin{enumerate}
    \item for any $T \in \tT$ and $F \in \fF$, $\Hom(T,F) = 0$, and 
    \item for any $E \in \dD$, there exist objects $T \in \tT$ and $F \in \fF$ with the following exact sequence in $\aA$:
    \[
      0 \rightarrow T \rightarrow E \rightarrow F \rightarrow 0.
    \]
  \end{enumerate}
 \end{defi}
 \begin{prop}[{\cite[Proposition 2.1.]{HRS:96}}]
  Let $\aA$ be the heart of a bounded t-structure on $\dD^b(X)$.
  If there exists a torsion pair $(\tT,\fF)$ in $\aA$, 
  then we have the following heart of a bounded t-structure on $\dD^b(X)$:
  \[
   \aA^{\sharp} = \langle \tT,\fF[1] \rangle_{\mathrm{ex}}.
  \] 
  Here, $\langle \tT,\fF[1] \rangle_{\mathrm{ex}}$ is the extension closure of $\tT$ and $\fF[1]$ in $\dD^b(X)$.
 \end{prop}
 We define twisted Chern characters for a class $B\in{N}^1(X)$.
 \begin{defi}
  Let $B$ be a class in ${N}^1(X)$.
  The twisted Chern character $\ch^B(E)$ of an object $E$ in $\dD^b(X)$ is defined to be
  \[
  \ch^B(E) = e^{-B} \ch(E) \in H^*(X;\bbQ).
  \]
 \end{defi}
We can compute
 \[
 \ch_0^B = \ch_0,\  \ch_1^B = \ch_1 - B\ch_0,\  \ch_2^B = \ch_2 - B\ch_1 + \frac{B^2}{2}\ch_0.
 \]
Let $H$ be an ample class in ${N}^1(X)$.
Then, we can define a twisted slope $\mu_{H}$ on $\Coh(X)$ by
\[
 \mu_{H}(E) = 
\begin{cases}
  \frac{\ch_1(E).H}{\ch_0(E)} &\text{ if }\ch_0(E) >0, \\
  \infty &\text{ if }\ch_0(E) =0.\\
\end{cases}
\]
We have the following torsion pair in $\Coh(X)$:
\begin{align*}
  \tT_{\beta,H} &= \langle E \colon E \text{ is $\mu_H$-stable with } \mu_H(E) > \beta \rangle_{\mathrm{ex}},\\
  \fF_{\beta,H} &= \langle E \colon E \text{ is $\mu_H$-stable with } \mu_H(E) \leq \beta \rangle_{\mathrm{ex}}     .
\end{align*}
We denote by $\aA_{\beta,H}$ the heart of the bounded t-structure on $\dD^b(X)$ associated with the torsion pair $(\tT_{\beta,H},\fF_{\beta,H})$.
We consider the following central charge.
\begin{align}
  Z_{B, H}(E) &=  -\ch_2^B(E) + \frac{H^2}{2} \ch_0 ^B(E) + i H \ch_1^B(E),\label{cc:bri}.
\end{align}
where $B$ is a divisor class and $H$ is an ample class in $N^1(X)$.

\begin{thm}[{\cite{ABL:13,Bri:08},\cite[Theorem 6.10.]{MS:17}}]
  Let $X$ be a smooth projective surface, and let $B$ and $H$ be classes in $N^1(X)$ with $H$ ample.
  Set $\beta = H.B$.
  Then, the pair $ \sigma_{B,H} = (Z_{B,H},\aA_{\beta,H})$ is a stability condition on $\dD^b(X)$.
\end{thm}
For any stability condition $\sigma_{B,H}$ constructed above, 
all skyscraper sheaves are $\sigma_{B,H}$-stable of phase $1$.
Such stability conditions are called \textit{normalized geometric stability conditions}.
Define the geometric chamber 
\[
 U(X) = \left\{ \sigma \in \Stab(X) \colon \text{ all skyscraper sheaves }\oO_x \text{ are $\sigma$-stable of one with } Z(\oO_x) = -1 \right\}.
\]
\begin{defi}
  Let $\sigma = (Z,\aA)$ be a stability condition on $\dD^b(X)$.
  We call $\sigma$ a normalized stability condition if $Z(\oO_x) = -1$.
  We denote by $\Stab_n(X)$ 
  the set of normalized stability conditions on $\dD^b(X)$.
  We write $\Stab_n^{\dagger}(X)$ for the connected component of $\Stab_n(X)$ containing the normalized geometric stability condition.
\end{defi}
In \cite{Del:23}, we can describe the normalized geometric chamber $U(X)$ by using the twisted Le Potier function.

\begin{defi}
  Let $X$ be a smooth projective surface, and let $B$ and $H$ be classes in $N^1(X)$ with $H$ ample.
  The twisted Le Potier function $L_{B,H}$ is defined to be
  \[
  \Phi_{H,B}(x) = \limsup_{\mu \rightarrow x} \left\{ \frac{\ch_2(E) - B.\ch_1(E)}{\ch_0(E)} \colon \mu_H\text{-stable }E \in \Coh(X) \text{ with }\mu_H(E) = \mu  \right\}.
  \]
\end{defi}
\begin{thm}[{\cite[Theorem 5.10]{Del:23}}] \label{thm:geom} 
  Let $X$ be a smooth projective surface and $(\alpha,\beta,B,H)$ be a quadruple in $\bbR \times \bbR \times N^1(X) \times \Amp(X)$. 
  Define a pair $\sigma_{\alpha - i\beta,B,H}=(Z_{\alpha-i\beta,B,H},\aA_{\beta,H})$
  \begin{align*}
    Z_{\alpha-i\beta,B,H} &=(\alpha-i\beta) \ch_0(E) + (B + i \omega)\ch_1(E) -\ch_2(E)\\
    \aA_{\beta,H} &= \langle \tT_{\beta,H},\fF_{\beta,H}[1] \rangle_{\mathrm{ex}}.
  \end{align*}
  If $\Phi_{H,B}(\alpha) > \beta$, then the pair $\sigma_{\alpha - i\beta,B,H} = (Z_{\alpha-i\beta,B,H},\aA_{\beta,H})$ is a stability condition in the normalized geometric chamber $U(X)$.
  Moreover, we have the following description of the geometric chamber:
  \[
   U(X) \cong \left\{(\alpha,\beta,B,H) \in \bbR \times \bbR \times N^1(X) \times \Amp(X) \colon \Phi_{H,B}(\alpha) > \beta \right\}.
  \]
\end{thm}

\subsection{Quantum cohomology and quantum connections}
We review quantum cohomology rings  and quantum differential equations.
A quantum cohomology ring is a deformation family of the cohomology ring of a smooth projective variety.
A quantum differential equation is a differential equation satisfied by the quantum cup product.
To define the quantum cohomology ring, we need to introduce the Gromov-Witten invariants.
For the proofs we refer the reader to \cite[Section 6-10.]{CK:99}.
\subsubsection{Gromov-Witten invariants and quantum cohomology}
Let $X$ be a smooth projective variety over $\bbC$.
\begin{defi}
  Let $X$ be a smooth projective variety, and $(C,p_1,\dots,p_n)$ be a curve of arithmetic genus $0$ with $n$ marked points.
  Consider a map $f\colon C \rightarrow X$.
  A tuple $(C,p_1,\dots,p_n,f)$ is called a stable map of genus $0$ with a homology class $\beta$ if the following conditions are satisfied.
  \begin{enumerate}
    \item The only singularities of $C$ are ordinary double points.
    \item $p_1,\dots,p_n$ are distinct smooth points in $C$.
    \item $C$ is of arithmetic genus $0$.
    \item $f_*([C]) = \beta$ in $ H_2(X;\bbZ)$.
    \item If $f$ is constant on a component $C_i \cong \bbP^1$of $C$, then the component must contain at least $3$ special points.
    \item If $f$ is constant on a component $C_i$ with genus one, then the component must contain at least one special point.
  \end{enumerate}
\end{defi}
For a class $\beta \in H_2(X;\bbZ)$, 
we denote by $\overline{M}_{0,n}(X,\beta)$ 
the moduli space of stable maps of genus $0$ 
with $n$ marked points and homology class $\beta$.
We have the following evaluation maps:
\begin{align*}
  ev_i\colon \overline{M}_{0,n}(X,\beta) &\rightarrow X,\\
  (C,p_1,\dots,p_n,f) &\mapsto f(p_i).
\end{align*}

Let $\alpha_1,\dots,\alpha_n$ be cohomology classes in $H^*(X;\bbC)$.
We get the Gromov-Witten invariants 
\[
\langle \alpha_1,\dots,\alpha_n \rangle_{0,n,\beta}^{\mathrm{GW}} = \int_{[\overline{M}_{0,n}(X,\beta)]^{\mathrm{vir}}} ev_1^*(\alpha_1) \cup \cdots \cup ev_n^*(\alpha_n).
\]
The Gromov-Witten invariants satisfy the following two properties introduced in \cite{CK:99}.
\begin{prop}[Degree axiom]
Let $\alpha_1,\dots,\alpha_n$ be cohomology classes in $H^*(X;\bbC)$.
If the  invariant $\langle \alpha_1,\dots,\alpha_n \rangle_{0,n,\beta}^{\mathrm{GW}}$ is non-zero, then we have
\[
\frac{1}{2}  \left\{\deg(\alpha_1) + \cdots + \deg(\alpha_n)  \right\}= \dim(X) + \int_{\beta}c_1(X) + n -3.
\]
\end{prop}
\begin{prop}[Divisor axiom]
  Let $\alpha_1,\dots,\alpha_n$ be cohomology classes in $H^*(X;\bbC)$.
  If $\alpha_1$ lies in $H^2(X;\bbC)$, then we have
  \[
  \langle \alpha_1,\dots,\alpha_n \rangle_{0,n,\beta}^{\mathrm{GW}} = \left(\int_{\beta} \alpha_1 \right) \langle\alpha_2,\dots,\alpha_n \rangle_{0,n-1,\beta}^{\mathrm{GW}}.
  \]
\end{prop}

\subsubsection{Quantum cohomology and quantum differential equations}
Next, we define the quantum cohomology ring on $X$.
Let $(-,-)$ be the Poincar\'e pairing on $H^*(X;\bbC)$.
The quantum cohomology is defined to be a family of commutative rings $(H^*(X;\bbC),\star_{\tau})$ parametrized by $\tau \in H^*(X;\bbC)$.
The product $\star_{\tau}$ is defined by the following relations:
\[
(\alpha \star_{\tau} \beta,\gamma) = \sum_{\beta \in H_2(X;\bbZ)} \sum_{n=0}^{\infty} \frac{1}{n!} \langle \alpha,\beta,\gamma,\tau,\dots,\tau \rangle_{0,n+3,\beta}^{\mathrm{GW}}.
\]
Since the Poincar\'e pairing is non-degenerate, the quantum cup product is well-defined.
By the divisor axiom, if $\tau$ lies in $H^2(X;\bbC)$, then we have
\[
(\alpha \star_{\tau} \beta,\gamma) = \sum_{\beta \in H_2(X;\bbZ)} \langle \alpha,\beta,\gamma\rangle_{0,3,\beta}^{\mathrm{GW}} e^{\int_{\beta} \tau}.
\]
If the quantum cup product $\alpha \star_{\tau} \beta$ is convergent for all $\alpha ,\beta \in H^*(X;\bbC)$, 
then the quantum cohomology ring is defined to be the ring $QH(X) = (H^*(X;\bbC),\star_{\tau})$.
For the quantum ring $QH(X)$, we can define the quantum differential equation.
Let $\xi(t)$ be a family of cohomology classes in $H^*(X;\bbC)$ parametrized by $t \in \bbR_{>0}$.
The quantum differential equation is defined to be
\begin{align}\label{eq:qde}
  t \frac{d \xi}{dt} = c_1(X) \star_{\log(t) c_1(X)} \xi.
\end{align}
In general, the sum does not converge, 
so we need to consider the truncated quantum cup product.

\subsubsection{Truncated quantum differential equations}
We review the truncated quantum differential equations introduced in \cite{HL:23}.
Let $\pi\colon X \rightarrow Y $ be a morphism between smooth projective varieties.
Let us consider a set $\mathrm{NE}(X/Y)_{\bbZ}$ consisting of 1-cycles in $X$ which are contracted by $\pi$.
\begin{defi}
  Let $d$ be a 1-cycle in $\mathrm{NE}(X/Y)_{\bbZ}$ with $c_1(X).d > 0$.
  The map $T_d \colon H^*(X;\bbC) \rightarrow H^*(X;\bbC)$ is defined by the equation:
  \[
  (T_d(\alpha),\beta) = \langle \alpha,\beta \rangle_{0,2,d}^{\mathrm{GW}}.
  \]
\end{defi}
The endomorphism $T_d$ is homogeneous of degree $2(1-c_1(X).d)$.
\begin{defi}
  Let $\psi = \omega + i B $ be a class in $H^2(X;\bbC)$ with a relatively ample class $\omega$.
  For a positive number $u$, we define the truncated quantum endomorphism $E_{\psi}(u) \colon H^*(X;\bbC) \rightarrow H^*(X;\bbC) $ by
  \[
  E_{\psi}(u)(-) = c_1(X) \cup (-) + \sum_{\substack{d \in {NE}(X/Y)_{\bbZ} \\ (c_1(X)-\omega).d> 0}} (c_1(X).d)u^{c_1(X).d}e^{-\psi.d}T_d.
  \]

\end{defi}
Note that it follows from \cite[Lemma 6.]{HL:23} that the sum is finite.
The \textit{truncated quantum differential equation} is defined to be
\[
t \frac{d \xi(t)}{dt} = -\frac{1}{z}E_{\psi}(t)(\xi(t)).
\]
Fix parameters $z \in \bbC^*$ and $\psi \in H^*(X;\bbC)$.
Consider a family of endomorphisms $\Phi_t \in \mathrm{End}(H^*(X;\bbC))$ such that 
$\Phi_t(\alpha)$ satisfies the truncated quantum differential equation for any class $\alpha$.
We mainly consider the following \textit{central charge} that is defined to be
\[
Z_t(\alpha) = \int_X \Phi_t(\alpha).
\]
The map is related to the Iritani's quantum central charge.
This is introduced in \cite{Iri:09} 
to define an integral structure on the quantum cohomology of a Fano variety.
In \cite{Iri:09,HL:23},
they consider the fundamental solution of the quantum differential equation.
We will discuss all solutions of the quantum differential equation in the case of the blowup of a surface in the next section.

\subsubsection{Quantum differential equations for a blowup surface}
  We compute the quantum differential equation in the case of the blowup of a surface.
  Let $f\colon X \rightarrow Y$ be a blowing up at a point of a smooth projective surface $Y$.

\begin{thm}[{\cite[Lemma 1.1.]{Hu:00}}]
    Suppose that at least one of $\alpha_i$($1 \leq i \leq m$) is the pullback
    of a cohomology class in $X$.
    Let $c$ be the class of the exceptional curve.
    Then,
    \[
      \langle \alpha_1,\cdots ,\alpha_m \rangle^{\mathrm{GW}}_{0,m,nc} = 0
    \]for any $n \in \bbZ$.
\end{thm}

This theorem shows that $T_{nC}$ vanishes on $H^0\oplus H^4(X;\bbC)$.
By the definition of the Poincare pairing, we have $(T_{nC}(\alpha),\beta) = 0 $ 
whenever $\deg(T_{nC}(\alpha))+\deg(\beta) \not= 4$.
Since it follows from the degree of $T_{nC}$ that $T_{nC}(H^2(X;\bbC)) \subset H^{4-2n}(X;\bbC)$, $T_{nC} = 0$ for all $n \not = 1$.
Since we can compute $\langle C,C\rangle_{0,2,C}^X =1$, the endomorphism $T_C$ satisfies the following conditions:
\begin{itemize}
	\item $T_C$ vanishes on $H^*(Y;\bbC)$, and 
	\item $T_C(C) = -C$.
\end{itemize}
Let $\psi = \omega + \sqrt{-1} B$ be a class in $H^2(X;\bbC)$ with a relatively ample class $\omega$.
Assume that $0 < \omega.C < 1$.
Write $\xi(t) = (\xi_0(t),f^*\xi_1(t) + \nu(t) C , \xi_2(t)) \in H^*(X;\bbC)$, where $\xi_1 $ lies in $H^2(Y;\bbC)$ , and $\nu \in \bbC$.
Then, we have the truncated quantum differential equation:
\begin{align*}
	t \frac{d\xi(t)}{dt} &=-\frac{1}{z}E_{\psi}(t)(\xi_t)\\
	&= \frac{1}{z} \left\{(0,(f^*K_Y + C)\xi_0,K_Y\xi_1 - \nu)+(0,e^{-\psi C}t \nu C,0)\right\} \\
	&= \frac{1}{z}(0,f^*K_Y\xi_0 +(\xi_0+e^{-\psi C}t \nu)C, K_Y\xi_1 - \nu).
\end{align*}
We thus get the following equation:
\begin{align}
	\frac{d\xi_0}{dt} &= 0, \label{eq:xi0}\\
	t \frac{d\xi_1}{dt} &= \frac{1}{z}K_Y \xi_0, \label{eq:xi1} \\
	t \frac{d\nu}{dt} &= \frac{1}{z}(\xi_0+e^{-\psi C}t \nu),  \label{eq:xi2}  \\
	t \frac{d\xi_2}{dt} &= \frac{1}{z}(K_Y \xi_1 - \nu). \label{eq:xi3}
\end{align}

\begin{prop}\label{prop:xieq}
  Let $\xi(t) = (\xi_0(t),\xi_1(t) + \nu(t) C , \xi_2(t)) \in H^*(X;\bbC)$ be a solution of the equations (\ref{eq:xi0})--(\ref{eq:xi3}).
  Then, there exist $C_0 \in \bbC$ and $C_1 \in H^2(X;\bbC)$ such that  $\xi_2(t)$ satisfies the following equation for $t \in \bbR_{\geq 1}$.
  \[
    t \frac{d}{dt} \left( t \frac{d \xi_2}{dt}\right) = \left\{ \left(K_Y ^2 -1 \right)\frac{C_0}{z} + e^{-\psi C}t (- \frac{K_Y^2C_0}{z} \log(t) + K_Y C_1 ) + e^{-\psi C}t^2\frac{d \xi_2}{dt} \right\}.
  \]
\end{prop}
\begin{proof}
  By the equation (\ref{eq:xi0}), the function $\xi_0(t)$ is constant.
  Let $\xi_0(t) = C_0$.
  There are a function $a(t)$ and a divisor class $D_t \in H^2(Y;\bbC)$ with $K_Y.D_t = 0$
  such that $\xi_1(t) = a(t)K_Y + D_t$.
  It follows from the equation (\ref{eq:xi1}) that $D_t$ is constant.
  By the equation (\ref{eq:xi1}), the function $a(t)$ satisfies the following differential equation:
  \[
  t \frac{d}{dt}a(t) = \frac{1}{z} C_0.
  \]
  Then, we have $a(t) = \frac{C_0}{z} \log(t) + a_1$ with $a_1 \coloneq a(1) \in \bbR$.
  Therefore, the function $\xi_1(t)$ is given by
  \[
  \xi_1(t) = \frac{C_0}{z} \log(t)K_Y + a_1K_Y + D_t.
  \]
  By the equation (\ref{eq:xi3}), we compute
  \begin{align*}
    t \frac{d}{dt} \left( t \frac{d \xi_2}{dt}\right) &=t \frac{d}{dt} \left( \frac{1}{z}(K_Y \xi_1 - \nu)\right)\\
    &= \frac{1}{z} \left\{t K_Y \frac{d}{dt}\xi_1 - t \frac{d}{dt}\nu \right\}\\
    &= \frac{1}{z} \left\{t K_Y \frac{d}{dt}\xi_1 - \frac{1}{z}(\xi_0 + e^{-\psi C}t \nu) \right\}\\
    &= \frac{1}{z} \left\{t K_Y \frac{d}{dt}\xi_1 - \frac{1}{z} \left(\xi_0 + e^{-\psi C}t (K_Y \xi_1 - t \frac{d \xi_2}{dt})\right) \right\}\\
    &= \frac{1}{z^2} \left\{ \left(K_Y ^2 -1 \right)\frac{C_0}{z} + e^{-\psi C}t \left(- \frac{K_Y^2C_0}{z} \log(t) + K_Y C_1 \right) + e^{-\psi C}t^2\frac{d \xi_2}{dt} \right\}.
  \end{align*}
  Therefore, the statement holds.
\end{proof}

Let $Z_t = \int_{X}\Phi_t$ be a map $H^*(X;\bbC) \rightarrow \bbC$, 
where $\Phi_t$ is a  solution of the quantum differential equation.
We call the map $Z_t$ \textit{the central charge}.
It follows from \Cref{prop:xieq} that the central charge satisfies
\[
	t\frac{d}{dt}\left(t \frac{d Z_t}{dt}\right) = \frac{1}{z^2}\left\{(K_Y^2-1)C_0+ e^{-\psi C}t \left(- \frac{K_Y^2C_0}{z} \log(t) + K_Y C_1 \right) + e^{-\psi C}t^2\frac{d Z_t}{dt} \right\}.
\]
Here, $C_0\in \bbC$ and $C_1 \in H^2(Y;\bbC)$ are constants.
Consider the case when $C_0,C_1$ are zero and z is one.
Then, we have
\begin{equation}\label{eq:debl}
  \frac{d}{dt}\left(t \frac{d Z_t}{dt}\right) = e^{-\psi C}t \frac{d Z_t}{dt}.
\end{equation}
The general solution of this differential equation can be written as
\begin{equation}\label{eq:cc0}
  Z_t = a \Ei(e^{-\psi C}t) + b,
\end{equation}
where $a$ and $b$ are constant and $\Ei(z)$ is the exponential integral.
The exponential integral $\Ei(z)$ is defined to be 
\[
  Ei(z) \coloneq -E_1(-z) = -\int_{-z}^{\infty}\frac{e^{-t}}{t} dt.
\]
To define this integral, we choice paths which does not cross $\bbR_{\leq 0}$, and we consider its principal values.
Since $\omega .C <1$, the parameter $e^{-\psi C}$ can take values from the following set:
\[
\left\{z \in \bbC \colon |z| \leq 1 \right\}.
\]
Therefore, we consider the following central charge for $\lambda \in \bbR \oplus i[-1,1]$:
\begin{equation}\label{eq:cc1}
  Z_t = a \Ei(\lambda t) + b.
\end{equation}

\section{Review: gluing hearts and stability conditions with respect to recollements}

In this section, we will review some of the facts on glued hearts with respect to recollements.
We summarize glued stability conditions defined in \cite{CP:09}.

\subsection{Gluing hearts}
The gluing construction of stability conditions was defined from \cite{CP:09} when a triangulated category $\dD$ has a semiorthogonal decomposition of two components $\langle \dD_1,\dD_2\rangle$.
In the general case of a semiorthogonal decomposition  $\langle \dD_1,\dD_2 ,\cdots, \dD_n \rangle$, 
glued stability conditions are constructed in \cite{HLJR}. 
First, we focus on glued hearts.

Let $\dD$ be a triangulated category with a semiorthogonal decomposition $\dD = \langle \dD_1,\dD_2\rangle$.
By definition, there is the following distinguished triangle for any object $E$:
\[
\rho_2(E) \rightarrow E \rightarrow \lambda_1(E) \overset{+}{\rightarrow}.
\]
Here, $\rho_2$ is the right adjoint functor of the inclusion $\dD_2 \rightarrow \dD$, 
and $\lambda_1$ is the left adjoint functor of the inclusion $\dD_1 \rightarrow \dD$.

\begin{lem}[{\cite[Lemma 2.1.]{CP:09}}]
Let $\dD = \langle \dD_1 ,\dD_2\rangle$ be a semiorthogonal decomposition
and $\aA_n$ be the hearts of t-structures on $\dD_n$.
If $\Hom^{\leq 0}(\aA_1,\aA_2) = 0$, then there exists a t-structure on $\dD$ with the heart
\[
  \aA_1 \circ \aA_2 \coloneq \left\{E \in \dD \mid \lambda_1(E) \in \aA_1,\rho_2(E) \in \aA_2 \right\}.
\]
\end{lem}

Assume that $D$ has the Serre functor $S$. Let $\dD_1$ and $\dD_2$ be admissible subcategories.
Let $\rR$ be a following recollement corresponding to the semiorthogonal decomposition $\dD = \langle\dD_1,\dD_2 \rangle $:
\[ 
  \rR \colon
  \begin{tikzcd}
  \dD_1 \arrow[r,"i_1" description] & \dD \arrow[l,bend right=40,"\lambda_1" description] \arrow[l,bend left= 40,"\rho_1"] \arrow[r,"\rho_2" description] & \dD_2 \arrow[l,bend right=40,"i_2" description ] \arrow[l,bend left= 40].
    \end{tikzcd}
\]
If there exist the hearts $\aA_n$ on $\dD_n$ for $n=1,2$, 
we have the heart $ \aA_1 *_{\rR} \aA_2$ by Theorem \ref{thm:BBD}.

\begin{lem}\label{lem:equiv}
  Let  $\aA_n$ be hearts on $\dD_n$ for $n=1,2$ satisfying $\Hom^{\leq 0}(i_1 (\aA_1),i_2 (\aA_2)) = 0$.
If $\dD_1$ and $\dD_2$ are admissible subcategories of $\dD$, then two hearts $ \aA_1 *_{\rR} \aA_2$ and $\aA_1 \circ \aA_2$ are the same.
\end{lem}
\begin{proof}
  It suffices to show that $\aA_1 \circ \aA_2 $ is contained in $\aA_1 *_{\rR} \aA_2$.
  Take $E \in  \aA_1 \circ \aA_2$. 
  We need to show the three conditions:
  \begin{itemize}
    \item $\rho_2(E) \in \aA_2$,
    \item $\Hom_{\dD}(E,i_1(A_1)[n]) =0$ for $n<0$,
    \item $\Hom_{\dD}(i_1(A_1)[n],E) = 0$ for $n>0$.
  \end{itemize}
  By definition, the first and second conditions hold.
  To see the third condition, there is a distinguished triangle
  \[
    i_2 \rho_2(E) \rightarrow E \rightarrow i_1\lambda_1(E) \overset{+}{\rightarrow}.
  \]
  Applying $\Hom_{\dD}(i_1(A_1)[n],-) $ yields a long exact sequence
  \[
  \cdots \rightarrow \Hom_{\dD}(i_1(A_1)[n],i_2 \rho_2(E)) \rightarrow \Hom_{\dD}(i_1(A_1)[n],E) \rightarrow \Hom_{\dD}(i_1(A_1)[n], i_1\lambda_1(E))  \rightarrow \cdots.
  \]
  By the definition of $\aA_1 \circ \aA_2$, $\lambda_1(E)$ lies in $\aA_1$.
  Thus,
    $\Hom_{\dD}(i_1(A_1)[n], i_1\lambda_1(E)) $ vanishes.
  It follows from the assumption that  $\Hom_{\dD}(i_1(A_1)[n], i_2\rho_2(E)) =0$.
  Therefore, $\Hom_{\dD}(i_1(A_1)[n],E) = 0$.
\end{proof}

The next lemma follows from the definition of the glued hearts.

\begin{lem}\label{lem:texact}
  Let $\rR$ be a recollement corresponding to $\dD = \langle \dD_1,\dD_2 \rangle$, and let $\aA = \aA_1 *_{\rR} \aA_2$.
  Then, the functor $i_1$ and $\rho_2$ are t-exact, and  $\lambda_1$ and $i_2$ are right t-exact.
  Assume that $\Hom^{\leq 0}(\aA_1,\aA_2) = 0$.
  Then, $i_2$ and $\lambda_1$ are t-exact.
\end{lem}

For any object $E \in \aA := \aA_1 *_{\rR} \aA_2$, there is the distinguished triangle
\[
  i_2 \rho_2(E) \rightarrow E \rightarrow i_1 \lambda_1(E) \overset{+}{\rightarrow}.
\]
If $\Hom^{\leq 0}(\aA_1,\aA_2) = 0$, since $i_1 \lambda_1(E)$ and $i_2 \rho_2(E)$ lie in $\aA$,
there exists the following exact sequence in $\aA$:
\[
  0 \rightarrow i_2 \rho_2(E) \rightarrow E \rightarrow i_1 \lambda_1(E)\rightarrow 0.
\]

\begin{lem}\label{lem:closedunder}
Let $\dD = \langle \dD_1,\dD_2 \rangle $ be a semiorthogonal decomposition 
and let $\rR$ be a corresponding recollement.
Then, $i_1(\aA_1)$ is closed under subobjects and quotients in $\aA := \aA_1 *_{\rR} \aA_2$.
If $\Hom^{\leq 0}(\aA_1,\aA_2) = 0$,
the subcategory
\[
  i_2(\aA_2) \subset \aA
\] is closed under subobjects and quotients.
\end{lem}
\begin{proof}
  We first prove that $i_1(\aA_1)$ is closed under subobjects.
  Let $E \in \aA_1$ and let $F$ be a subobject of $i_1(E)$ in $\aA$.
  Applying the t-exact functor $\rho_2$ gives an injection $\rho_2(F) \hookrightarrow \rho_2(i_1(E))$.
  Since $\rho_2(i_1(E)) = 0$, $F$ lies in $ i_1(\aA_1)$.
  By a similar argument, we see the first statement.
  The same argument works for the second statement.
\end{proof}

The following proposition will be needed in \Cref{sec:stabilityCk}.
\begin{prop}\label{prop:torpair}
Let $\dD = \langle \dD_1,\dD_2\rangle$ be a semiorthogonal decomposition and let $\rR$ be the corresponding recollement. 
For any hearts $\aA_1 \subset \dD_1$ and $\aA_2 \subset \dD_2$,
consider a torsion pair $(\tT_2,\fF_2)$ in the heart $\aA_2\subset \dD_2$.
Assume that there exists a torsion pair $(\tT,\fF)$ in the glued heart $\aA_1 *_{\rR}\aA_2$ satisfying 
the following:
\begin{enumerate}
  \item $\aA_1 \subset \tT$, and 
  \item $\rho_2(\tT) \subset \tT_2$ and $\rho_2(\fF) \subset \fF_2$.
\end{enumerate}
Set $\aA^{\#} = \langle \tT ,\fF[1]\rangle_{\mathrm{ex}}$ and ${\aA_2}^{\#} = \langle \tT_2 ,\fF_2[1]\rangle_{\mathrm{ex}}$.
Then, $\aA^{\#}$ is the  heart glued from $\aA_1$ and ${\aA_2}^{\#}$.
\end{prop}
\begin{proof}
  Set $\aA = \aA_1 *_{\rR} {\aA_2}$ and $\bB = \aA_1 *_{\rR} {\aA_2}^{\#}$.
  It suffices to show that both $\tT$ and $\fF[1]$ are contained in $\bB$.
  We need to show the following conditions.
  \begin{enumerate}
    \renewcommand{\labelenumi}{(\roman{enumi})}
    \item $\rho_2(\aA^{\#} ) \in \aA_2^{\#}$.
    \item $\Hom_{\dD}(\tT,i_1(\aA_1)[n]) =0$ for $n<0$, and $\Hom_{\dD}(i_1(\aA_1)[n],\tT) = 0$ for $n>0$.
    \item $\Hom_{\dD}(\fF[1],i_1(\aA_1)[n]) =0$ for $n<0$, and $\Hom_{\dD}(i_1(\aA_1)[n],\fF[1]) = 0$ for $n>0$.
  \end{enumerate}
 By the assumption (2) and t-exactness of $\rho_2$,
 the first condition is true.
 The condition (ii) is trivial since $\tT \subset \aA$.
 It follows from the assumption (1) that  $\Hom_{\dD}(i_1(\aA_1)[n],\fF[1]) = 0$.
\end{proof}

Next, we will discuss the Noetherianity of a glued heart.
\begin{prop}\label{prop:noeth}
  Assume that $\aA_1$ and $\aA_2$ are Noetherian.
  If $\Hom^{<0}(\aA_1,\aA_2) = 0$, then the glued heart $\aA_1 *_{\rR} \aA_2$ is also Noetherian.
\end{prop}
\begin{proof}
  Set $\aA = \aA_1 *_{\rR} \aA_2$.
  Suppose there exists an infinite ascending chain 
  \[
    A_0 \subset A_1 \subset \cdots \subset A_i \subset \cdots \subset A_n =A
  \]
  of subobjects of $A \in \aA$.
  For any $i$, there exists a distinguished triangle
  \[
    i_2 \rho_2 A_i \rightarrow A_i \rightarrow i_1 \lambda_1 A_i \overset{+}{\rightarrow} i_2 \rho_2 A_i[1].
  \]
Since the functor $i_2$ is t-exact, we have the following exact sequence in $\aA$:
\[
  0 \rightarrow i_1 \hH^{-1}(\lambda_1 A_i) \rightarrow i_2 \rho_2 A_i \rightarrow A_i \rightarrow i_1 \hH^0(\lambda_1 A_i) \rightarrow 0.
\]
Additionally, there exists the sequence
\[
  i_2 \rho_2 A_0 \subset \cdots \subset i_2 \rho_2 A_n.
\]
Since $\aA_2$ is Noetherian, the above sequence terminates.
We can assume that $i_2 \rho_2 A_i$ does not depend on $i$.

Using the snake lemma,
 all induced morphisms $i_1 \hH^{-1}(\lambda_1 A_i) \rightarrow i_1 \hH^{-1}(\lambda_1 A_{i+1})$ are injective.
Since there is an ascending sequence $i_1 \hH^{-1}(\lambda_1 A_0) \subset \cdots \subset i_1 \hH^{-1}(\lambda_1 A_n) \subset A$,
$i_1 \hH^{-1}(\lambda_1 A_i)$ is also independent from $i$.
Similarly, $i_1 \hH^{0}(\lambda_1 A_i)$ does not depend on $i$.
We conclude that the ascending chain of $A$ terminates.
\end{proof}

\subsection{Gluing stability conditions}
From now on, we fix a homomorphism $v\colon K_0(\dD) \rightarrow \Lambda$.
We also assume that there is a splitting $\Lambda = \Lambda_1 \oplus \Lambda_2$, where $\Lambda_i = v(\iota_n(K_0(\dD_i)))$.
Here, $\iota_n$ is the natural inclusion $K_0(\dD_i) \rightarrow K_0(\dD)$.
To simplify notation, we use $\Stab(\dD)$, $\Stab(\dD_1)$ and $\Stab(\dD_2)$ instead of $\Stab_{\Lambda}(\dD)$, $\Stab_{\Lambda_1}(\dD_1)$ and $\Stab_{\Lambda_2}(\dD_2)$ .

\begin{defi}[{\cite{CP:09}}]
  Suppose we have $\sigma_1=(Z_1,\aA_1) \in \Stab(\dD_1)$ and $\sigma_2 = (Z_2,\aA_2) \in \Stab(\dD_2)$.
  We say that a (pre)stability condition $\sigma = (Z,\aA)$ is glued from $\sigma_1$ and $\sigma_2$ with respect to a recollement $\rR$ 
  if 
  \begin{enumerate}
    \item for any $E \in \dD_n$, $Z(i_n(E)) = Z_n(E)$,
    \item $\aA = \aA_1 *_{\rR} \aA_2$.
  \end{enumerate}
\end{defi}
Write a glued stability condition as $\sigma = \sigma_1 *_{\rR} \sigma_2$ if it exists.
We note that $\Hom ^{\leq 0}(\aA_1,\aA_2)$ might be non-trivial
even if there exists a glued stability condition $\sigma = \sigma_1 *_{\rR} \sigma_2$,
since there exists a stability condition whose heart is not degenerated (i.e. $\Hom^{\leq0}(\aA_1,\aA_2)\not=0$).
We will give an example of such a stability condition in Theorem \ref{thm:nondegenerate}.
We call a glued stability condition $\sigma = \sigma_1 *_{\rR} \sigma_2$ with $\Hom^{\leq0}(\aA_1,\aA_2)=0$ an $\rR$-glued stability condition.

\begin{thm}[{\cite[Theorem 3.6.]{CP:09}}]\label{thm:glued}
Let $\dD = \langle \dD_1,\dD_2 \rangle $ be a semiorthogonal decomposition 
and let $\sigma_1 = (Z_1,\pP_1)$ and $\sigma_2 = (Z_2,\pP_2)$ be stability conditions on $\dD_1$ and $\dD_2$.
Assume that both stability conditions are reasonable.
We assume the following conditions:
\begin{enumerate}
  \item $\Hom^{\leq 0 }(\pP_1(0,1],\pP_2(0,1]) =0$;
  \item  There exists a real number $a \in (0,1)$ such that $\Hom^{\leq 0}(\pP_1(a,a+1],\pP_2(a,a+1]) = 0$.
\end{enumerate}
Then, there exists a pre-stability condition $\sigma$ glued from  $\sigma_1$ and $\sigma_2$. Futhermore, $\sigma$ is reasonable.
\end{thm}
\begin{rmk}
 We call $\sigma = (Z,\pP)$ on $\dD$ a reasonable stability condition if it satisfies 
 \[
  \inf_{E \text{ is semistable with } E \not= 0} |Z(E)| > 0.
 \]
 If $\sigma$ is a stability condition satisfying the support property with respect to a lattice $\Lambda$,
 the stability condition $\sigma$ is reasonable.
\end{rmk}

Next, we will discuss the support property of a pre-stability condition $\sigma$ glued from $\sigma_1$ and $\sigma_2$.
The propositions were motivated by \cite[Lemma 4.6.,Proposition 4.7.]{Tod:14}.

\begin{lem}\label{lem:stable}
  Let $\sigma = (Z,\aA)$ be a pre-stability condition glued from pre-stability conditions $\sigma_1 = (Z_1,\aA_1) $ and $\sigma_2= (Z_2,\aA_2)$ on $\dD_1$ and $\dD_2$.
  \begin{enumerate}
    \item $E \in \aA_1$ is $\sigma_1$-(semi)stable if and only if $i_1(E)$ is $\sigma$-(semi)stable.
    \item Additionally, assume $\Hom^{\leq0}(\aA_1,\aA_2) = 0$.  The object $i_2(M)$ is $\sigma$-(semi)stable if and only if $M \in \aA_2$ is $\sigma_2$-(semi)stable,
  \end{enumerate}
  \end{lem}
\begin{proof}
  Both statements follow from Lemma \ref{lem:closedunder}.
\end{proof}
Next, we will discuss the support property of pre-stability condition $\sigma$ glued from $\sigma_1 = (Z_1,\aA_1) $ and $\sigma_2 = (Z_2,\aA_2)$.

\begin{prop}\label{prop:support}
  Let  $\sigma_1 = (Z_1,\aA_1)$ and $\sigma_2= (Z_2,\aA_2)$ be stability conditions on $\dD_1$ and $\dD_2$ respectively.
  We assume that $\Hom^{\leq 0 }(\aA_1,\aA_2) =0$.
  Assume that there exists the pre-stability condition $\sigma = (Z,\aA)$ glued from $\sigma_1$ and $\sigma_2$.
  If there exists $0 < \theta \leq 1$ such that $Z(\aA_2) \subset \bbH_{\theta}$, 
  where $\bbH_{\theta}$ is the set defined by 
  \[
    \left\{ r \exp(i \pi \phi) \mid r \in \bbR_{>0}, \phi \in [\theta,1]\right\},
  \]
  then the glued pre-stability condition $\sigma = (Z,\aA)$ satisfies the support property.
\end{prop}
\begin{proof}
  Let $\rR$ be a recollement corresponding to $\dD = \langle \dD_1,\dD_2 \rangle$.
  \[
  \rR \colon 
  \begin{tikzcd}
    \dD_1 \arrow[r,"i_1" description] & \dD \arrow[l,bend right=40,"\lambda_1" description] \arrow[l,bend left= 40,"\rho_1"] \arrow[r,"\rho_2" description] & \dD_2 \arrow[l,bend right=40,"i_2" description ] \arrow[l,bend left= 40]
      \end{tikzcd}.
  \]
By the assumption, there exists a constant $K(\theta)$ such that it satisfies 
\[
  \frac{|z_1 + \cdots +z_k|}{|z_1|+ |z_2|+ \cdots + |z_k|} \geq K(\theta)
\]
for any $k \geq 1$ and $z_1,\dots,z_k \in \bbH_{\theta}$.
Take a $\sigma$-semistable object $E$ in $\aA$.
By \Cref{lem:texact}, there exists the following exact sequence in $\aA$:
\begin{align}\label{exact:glued}
0 \rightarrow i_2 \rho_2(E) \rightarrow E \rightarrow i_1 \lambda_1(E) \rightarrow 0. 
\end{align}
We consider the following three cases:
\begin{enumerate}
  \item $i_2 \rho_2(E) = 0$,
  \item $i_1 \lambda_1(E) =0$,
  \item $i_2 \rho_2(E) $ and $i_1 \lambda_1(E)$ are non-zero.
\end{enumerate}
In the first case, $E$ lies in $i_1(\aA_1)$.
By the support property of $\sigma_1$ and Lemma \ref{lem:stable}, there exists a constant $C_1$ such that
\[
\frac{|| E ||}{|Z(E)|} < C_1.
\]
In the second case, by the same reasons, we have a constant $C_2$ such that
\[
\frac{|| E ||}{|Z(E)|} < C_2.
\]

In the third case, consider the HN factors of $i_2 \rho_2(E)$ and $i_1 \lambda_1(E)$.
We show that $Z(A)$ lies in $\bbH_{\theta}$ for any HN factor $A$ of $i_1 \lambda_1(E)$.
Suppose that there exists an exact sequence
\[
0 \rightarrow E_1 \rightarrow \lambda_1(E) \rightarrow A_1 \rightarrow 0
\]
with 
\[
\arg Z(E_1) > \arg Z(\lambda_1(E)) > \arg Z(A_1).
\]
Here, we assume that $A_1$ is stable with respect to $\sigma_1$.
By the exact sequence (\ref{exact:glued}), it follows that 
\[ 
   \pi \theta \leq \arg Z(i_2 \rho_2(E)) < \arg Z(E)< \arg Z(i_1 \lambda_1(E)) < \arg Z_1(E_1).
\]
Since there is  the surjection $E \twoheadrightarrow i_1 \lambda_1(E) \twoheadrightarrow i_1 (A_1)$ and $i_1(A_1)$ is stable with respect to $\sigma$,
we have $\pi \theta <\arg Z(E)< \arg Z(i_1(A_1))$
Therefore, $Z(A)$ lies in $\bbH_{\theta}$ for any HN factor $A$ of $i_1 \lambda_1(E)$.
Let us take the HN factors $A_1,\dots,A_k$ of $i_1 \lambda_1(E)$  and 
$F_1,\dots,F_l$ of $i_2 \rho_2(E)$.
Set $C = \max \{C_1,C_2\}$.
Then, it follows that
\begin{align*}
  \frac{||E||}{|Z(E)|} &\leq \frac{\sum ||F_i|| + \sum ||A_j||}{|\sum Z(F_i) + \sum Z(A_j)|}\\
  &\leq \frac{1}{K(\theta)}\frac{\sum ||F_i|| + \sum ||A_j||}{\sum |Z(F_i)| + \sum |Z(A_j)|}\\
  &\leq \frac{C}{K(\theta)}.
\end{align*}
Therefore, the glued pre-stability condition $\sigma$ satisfies the support property.
\end{proof}

Next, we consider the case where $Z(\aA_1)$ lies in $\bbH \backslash \bbH_{\theta}$ for some $0<\theta<1$.
Similar arguments show the following proposition.
\begin{prop}
If $\sigma_1$ and $\sigma_2$ satisfy the support property, and
if there exists $0 < \theta \leq 1$ such that $Z(\aA_1) \subset \bbH \backslash  \bbH_{\theta}$,
then  the glued pre-stability condition $(Z,\aA)$ also satisfies the support property.
\end{prop}

\section{Gluing regions of stability conditions of blowup surfaces} \label{sec:gluingregion}
In this section, we summarize derived categories and stability conditions on the blowing up of surfaces.
It is shown that stability conditions constructed in \cite{Tod:14} are obtained from glued stability conditions.

\subsection{Semiorthogonal decomposition corresponding to a blowing up} \label{sec:sodbl}
From now on, we focus on the case where $f \colon X \rightarrow Y$ is a blowing up a point $o \in Y$.
Then, there is the following diagram:
\[
  \begin{tikzcd}
    C \arrow[d,"\pi"] \arrow[r,"j"] & X \arrow[d,"f"]\\
    o \arrow[r] & Y
  \end{tikzcd}
\]
Here, $C$ is an exceptional curve.
Since the derived category of one point is equivalent to the derived category of the finite dimensional vector space over $\bbC$,
we write $\dD^b(\Vect)$ instead of the derived category of a point.
In this case, we have a semiorthogonal decomposition corresponding to the above recollement.
\begin{thm}[{\cite{Orl:93}}]\label{thm:sodblowup}
Let $\Phi_k$ be the functor $\dD^b(\Vect) \rightarrow \dD^b(X)$ defined by
\[
  \Phi_k(V) =j_* (\pi^*(V)\otimes \oO_C(k))
\]
for $k \in \bbZ$.
Then, there is a semiorthogonal decomposition:
\[
\dD^b(X) = \langle \Phi_{-1}(\dD^b(\Vect)), \dD^b(Y) \rangle.
\]
\end{thm}
Since $\Phi_{-1}(\dD^b(\Vect))$ is equivalent to a subcategory generated by $\oO_C(-1)$ in $\dD^b(X)$,
we denote by $\langle \langle \oO_C(-1)\rangle ,\dD^b(Y) \rangle$ the semiorthogonal decomposition in the above theorem.
We have the following recollements for each $k \in \bbZ$:
\[
  \rR_{k+1} \colon
  \begin{tikzcd}
  \dD^b(Y) \arrow[r,"f_{k+1}^L" description] & \dD^b(X) \arrow[l,bend right=40,"f_{k+2}" description] \arrow[l,bend left= 40,"f_{k+1}"] \arrow[r,"\rho_{k+1}" description] & \langle \oO_C(k+1) \rangle  \arrow[l,bend right=40,"j" description ] \arrow[l,bend left= 40],
    \end{tikzcd}
\]
\[
  \lL_k \colon
  \begin{tikzcd}
  \langle \oO_C(k) \rangle \arrow[r,"j" description] & \dD^b(X) \arrow[l,bend right=40,"\rho_{k+1}" description] \arrow[l,bend left= 40,,"\rho_{k}" description] \arrow[r,"f_{k+1}" description] & \dD^b(Y) \arrow[l,bend right=40,"f_{k+1}^L" description ] \arrow[l,bend left= 40,"f_k^L" description].
    \end{tikzcd}
\]
Here, $f_k = \bbR f_*(\otimes \oO(kC))$, and $f^L_k = f^*(-) \otimes \oO_C(-kC)$ is the left adjoint functor of $f_k$.
The functor $\rho_k$ is the right adjoint functor of the inclusion functor $i_k \colon \langle \oO_C(k) \rangle \hookrightarrow \dD^b(X)$.
It follows from the construction of \Cref{lem-defi:actiononauto} that $\lL_k \otimes \oO(C) = \lL_{k-1}$.
The recollement $\lL_k$ is the same as the lower recollement of $\rR_{k+1}$.

\subsection{Three types of stability conditions on a blowup surface $X$}
For any surfaces, there exists geometric stability conditions defined in Section \ref{pre:stab}.
We will introduce three other types of stability conditions using the semiorthogonal decompositions in Theorem \ref{thm:sodblowup}.
We will compare these stability conditions with stability conditions constructed in \cite{Tod:13,Tod:14}.
We consider the three such glued hearts:
\begin{itemize}
  \item[($C_k$)] $\langle \oO_C(k)[1]\rangle *_{\lL_k} \aA_Y$,
  \item[($\rR_{k+1}$)] $ \aA_Y*_{\rR_{k+1}} \langle \oO_C(k+1)[n]\rangle $ for $n<0$,
  \item[($\lL_k$)]  $\langle \oO_C(k)[n]\rangle *_{\lL_k} \aA_Y$ for $n>1$.
\end{itemize}
Here, $\lL_k$ and $\rR_{k+1}$ are recollements defined in Section \ref{sec:sodbl}, and
$\aA_Y$ denotes the heart of $\dD^b(Y)$ discussed in Section \ref{pre:stab}.

\subsubsection{Stability conditions with hearts of type $(C_k)$.}\label{sec:stabilityCk}
First, we discuss the stability conditions whose hearts  are of type $(C_k)$.
These stability conditions appear in \cite{Tod:13,Tod:14}.
In general cases, they are discussed in \cite{TX:17}.
We consider the following recollements 
\[
  \lL_k \colon
    \begin{tikzcd}
    \langle \oO_C(k)\rangle \arrow[r,"j"] & \dD^b(X)  \arrow[r,"f_{k+1}"] & \dD^b(Y) ,
      \end{tikzcd}
\] where $f_k = \bbR f_*(-\otimes \oO(kC))$ for $k \in \bbZ$.
Let $B$ be a divisor class in $N^1(X)$ and $\omega$ be an ample divisor on $Y$. 
We assume that $k- \frac{1}{2} < B.C < k+\frac{1}{2}$.
We consider the glued heart defined by
\[
  \aA_{B,f^* \omega} = \langle \oO_C(k)[1]\rangle *_{\lL_k} \aA_{f_*B,\omega}.
\]
The heart can be defined in two other ways.
First, this can be obtained by tilting of the perverse heart $\Per(X \backslash Y)$.
Here, the heart $\Per(X \backslash Y)$ in $\dD^b(X)$ is the glued heart $\langle \oO_C(-1)[1]\rangle *_{\lL_{-1}} \Coh(Y)$,
 and this is essentially studied in \cite{Bri:02}.
We define the slope function $\mu_{B,f^* \omega}(E)$ to be 
\[
  \mu_{B,f^* \omega}(E) = \begin{cases}
    \frac{\ch_1^B(E) .f^* \omega}{\ch_0(E)} \text{ if }\ch_0(E)\not= 0 \\
    \infty  \text{ if }\ch_0(E)= 0.\\
  \end{cases} 
\]
Then, we define subcategories in $\Per(X \backslash Y)$ by
\begin{align*}
  \tT &= \excl{E \in \Per(X \backslash Y) \colon E \text{ is $\mu_{B,f^* \omega}(E)-$semistable with } \mu_{B,f^* \omega}(E) >0 }\\
  \fF &= \excl{E \in \Per(X \backslash Y) \colon E \text{ is $\mu_{B,f^* \omega}(E)-$semistable with } \mu_{B,f^* \omega}(E) \leq0 }.\\
\end{align*}
\begin{prop}[{\cite[Lemma 3.6.]{Tod:13}}]
The pair $(\tT,\fF)$ is a torsion pair on $\Per(X \backslash Y)$.
\end{prop}

Let $\aA_{B,f^* \omega}^{\mathrm{per}}$ denote the tilting heart with respect to the above torsion pair.
It follows from Grothendieck-Riemann-Roch theorem that 
\[
  \mu_{B,f^*\omega}(E) = \mu_{f_*B,\omega}(f^*(E))
\] for all $E$ with $f_*(E) \not = 0$.
Applying Proposition \ref{prop:torpair}, we conclude that $\aA_{B,f^* \omega} = \aA_{B,f^* \omega}^{\mathrm{per}}$ if $-\frac{1}{2} < B.C < \frac{1}{2}$.

Next, we consider another tilting heart to construct the glued heart.
The following construction appears in \cite{Bri:08,TX:17}.
The class $f^* \omega$ is not an ample divisor.
However, as in the case of ample divisor, we can define $\mu_{f^*\omega}$-stability on $\Coh(X)$.
For $a \in \bbR$, we have a torsion pair $({\tT^a}_{f^*\omega},{\fF^a}_{f^*\omega})$ defined to be
\begin{align*}
  \tT^a_{f^*\omega} &=  \excl{E \in \Coh(X) \colon E \text{ is semistable with } \mu_{f^* \omega}(E) >a }\\
  \fF^a_{f^*\omega} &= \excl{E \in \Coh(X) \colon E \text{ is semistable with } \mu_{f^* \omega}(E) \leq a }.
\end{align*}
Let $\aA^a_{f^*\omega}$ be a tilting heart obtained from the above torsion pair.
\begin{lem-defi}[{\cite[Lemma 3.2.]{TX:17}}]
  The pair of subcategories
  \begin{align*}
    \tT^a _k &= \left\{ E \in\aA^a_{f^*\omega} \colon \Hom(E, \oO_C(l))= 0 \text{ for }l\leq k \right\} \\
    \fF^a_k &= \excl{\oO_C(l) \colon l \leq k}
  \end{align*}
  defines a torsion pair on $\aA^a_{f^*\omega}$ for all $k\in \bbZ$.
  Let $\bB^a_{k,f^* \omega}$ denotes the tilting heart.
\end{lem-defi}

\begin{lem}[{cf.\cite[Section 4.]{TX:17}}]
  Let $B$ be a class in $N^1(X)$ and $\omega$ be an ample class on $Y$.
  Assume $-\frac{1}{2} < B.C < \frac{1}{2}$.
  Then, $\bB^{B.f^*\omega}_{-1,f^* \omega} =  \aA_{B,f^* \omega}^{\mathrm{per}}=\aA_{B,f^* \omega}$.
\end{lem}
Using this lemma we can compare $\bB^{B.f^*\omega}_{k,f^* \omega}$ with $\aA_{B,f^* \omega}$ in the general cases.

\begin{prop}
  Let $B$ be a class in $N^1(X)$ and $\omega$ be an ample class on $Y$.
  Assume $k-\frac{1}{2} < B.C < k+\frac{1}{2}$ for some integer $k$.
  Then, $\bB^{B.f^*\omega}_{k-1,f^* \omega} = \aA_{B,f^* \omega}$.
\end{prop}
\begin{proof}
  Applying the equivalence $-\otimes\oO(-kC)$ gives
  \[
    \aA_{B,f^* \omega}\otimes \oO(-kC) = \aA_{B+kC,f^* \omega}.
    \]
  Since $-\frac{1}{2} < (B+kC).C <+ \frac{1}{2}$,
  it follows from the above lemma that $\aA_{B+kC,f^* \omega} = \bB^{(B+kC).\omega}_{-1,f^* \omega}$.
  Since $(B + kC).f^*\omega = B.f^*\omega$, the heart $\bB^{(B+kC).f^*\omega}_{k-1,f^* \omega}$ is the same as  $\bB^{B.f^*\omega}_{k-1,f^* \omega}$.

  Thus, It holds that
  \[ 
    \bB^{(B+kC).f^*\omega}_{-1,f^* \omega} \otimes \oO(kC) = \bB^{(B+kC).f^*\omega}_{k-1,f^* \omega} = \bB^{B.f^*\omega}_{k-1,f^* \omega}.\]

\end{proof}

\begin{thm}[{\cite[Theorem 5.4.]{TX:17}}]\label{thm:nondegenerate}
  Let $B$ be a class in $N^1(X)$ and $\omega$ be an ample class on $Y$.
  If $B.C \not \in \bbZ + \frac{1}{2}$,   
  then the pair $(Z_{B,f^*\omega},\aA_{B,f^*\omega})$ defines a stability condition.
\end{thm}

\begin{defi}
  Let $f\colon X \rightarrow Y$ be a blowing up of a point $o \in Y$ and $C$ be an exceptional curve.
  The stability condition $\sigma \in \Stab(X)$ is $(C_k)$-type 
  if $\sigma$ lies in the boundary of $U(X)$ and the Jordan-H\"older filtration of $\oO_x$ is given by
  \[
    0 \rightarrow \oO_C(k+1) \rightarrow \oO_x \rightarrow \oO_C(k)[1]\rightarrow 0.
  \]
\end{defi}
The stability condition $(Z_{B,f^*\omega},\aA_{B,f^*\omega})$ is $(C_k)$-type if $k - \frac{1}{2} <B.C <k+\frac{1}{2}$.
Indeed, two objects $\oO_C(k+1)$ and $\oO_C(k)[1]$ are $\sigma$-stable objects of phase one in $\aA_{B,f^*\omega}$.
One can move the stability condition $\sigma = (Z_{B,f^*\omega},\aA_{B,f^*\omega})$ to a stability condition $\tau$ whose heart $\aA_{\tau}$ contains $\oO_C(k+1)$ and $\oO_C(k)$.
If we take $\tau$ sufficiently close to $\sigma$,  $\tau$ is a geometric stability condition.

The following proposition will be needed to compare the geometric chamber and a gluing region.
\begin{prop}
  Let $B$ be a class in $N^1(X)$ and $\omega$ be an ample class on $Y$, and $(Z_{B,f^*\omega},\aA_{B,f^*\omega})$ be a $(C_k)$-type stability condition.
  Consider $Z_Y (-) = Z_{B,f^*\omega}(f^*(-))$ and $\aA_Y = \aA_{f_*B,\omega}$.
  Then, the pair $(Z_Y,\aA_Y)$ is a geometric stability condition on $Y$.
\end{prop}
\begin{proof}
  First, we compute the central charge $Z_Y$.
  We have 
  \[
  Z_Y (E) = -\ch_2(E) + (f_*B + i\omega).\ch_1(E) - \frac{(B +i f^*\omega)^2}{2} \ch_0(E).
  \]
  By \Cref{thm:geom}, it is enough to show that $Z_Y(F)$ is negative for any $F\in \aA_Y$ such that $\Imm Z_Y(F) = 0$.  
  Since $f^*\aA_Y $ lies in $\aA_{B,f^*\omega}$, we have the inequality.
\end{proof}

\begin{rmk}
  Let $\aA_1$ and $\aA_2$ be hearts on $\dD_1$ and $\dD_2$.
  Consider a heart $\aA$ containing $\aA_1$ and $\aA_2$ on $\dD = \langle \dD_1,\dD_2 \rangle$.
  If $\Hom^{\leq 0}(\aA_1,\aA_2) = 0$, then the heart $\aA$ is the same as the glued heart $\aA_1 \circ \aA_2$.
  However, if $\Hom^{\leq 0}(\aA_1,\aA_2) \not = 0$, the heart $\aA$ might not be the glued heart.

  For example, consider the case where $\dD_1 = \langle \oO_C(-1)\rangle$ and $\dD_2 = \dD^b(Y)$ for a blowing up of surface $f \colon X \rightarrow Y$.
  Let $\aA_Y$ be a heart on $\dD^b(Y)$ which contains all skyscraper sheaves $\oO_y$ for any closed point $y \in Y$.
  Set $\aA =  \aA_{Y} *_{\rR_{0}} \langle \oO_C\rangle$ and $\bB = \langle \oO_C(-1)[1] \rangle *_{\rR_{-1}} \aA_Y$.
  A simple computation shows that $\bB$ contains $\oO_C$ and $\aA_Y$.
  However, the hearts $\aA$ and $\bB$ are different.
  Indeed, $\aA_{f^*\omega}$ contains $\oO_C(-1)[1]$.
  On the other hand, the heart $\aA$ does not contain $\oO_C(-1)[1]$.
\end{rmk}

\subsubsection{Stability conditions with hearts of type $(\rR_{k+1})$.}
We discuss the $\rR_{k+1}$-glued stability conditions.
Since the space of stability conditions on a point $\Spec \bbC$ are isomorphic to $\bbC$,
there is a stability condition $\tau_\lambda = (W_\lambda,Q_\lambda)$ for any $\lambda \in \bbC^*$ defined by the following conditions:
\begin{enumerate}
  \item the central charge $W_\lambda$ sends the class $[\bbC]$ to $e^{\lambda}$, and
  \item the subcategory $Q_{\lambda}(\frac{\Imm \lambda}{\pi})$ of the slicing $Q_{\lambda}$ contains the object $\bbC$.
\end{enumerate}

Write the recollement
\[
  \rR_{k+1} \colon
    \begin{tikzcd}
      \dD^b(Y) \arrow[r,"f_{k+1}^L" description] & \dD^b(X) \arrow[l,bend right=40,"f_{k+2}" description] \arrow[r,"\rho_{k+1}"] & \langle \oO_C(k+1)\rangle \arrow[l,bend right=40,"\rho_{k+1}^L" description].
      \end{tikzcd}
\]
Here, $f_{k+1}^L$ is the left adjoint functor of $f_{k+1}$.
\begin{prop}\label{prop:Rprimeglued}
Let $\lambda$ be a complex number with $\Imm \lambda > \pi$, 
and $\sigma_Y= (Z_Y,\aA_Y)$ be a normalized geometric stability condition on $Y$.
Then, there exists the glued stability condition $\sigma_{Y} *_{\rR_{k+1}} \tau_{\lambda}$.
\end{prop}
\begin{proof}
  We will check the conditions in Theorem \ref{thm:glued}.
  First, we show the condition (1), i.e. 
  \[\Hom^{\le 0}(f_{k+1}^L\aA_Y,\rho_{k+1}^L(Q_{\lambda}(0,1]) )= 0.\]
  Let $n$ be an integer such that $0 < \frac{\Imm \lambda}{\pi} + n \leq 1$.
  We have $\rho_{k+1}^L(Q_{\lambda}(0,1]) = \langle \oO_C(k+1)[n]\rangle$.
  It follows from  adjunction that 
  \begin{align*}
    \Hom^{\le 0}(f_{k+1}^L\aA_Y, \rho_{k+1}^L Q_{\lambda}(0,1]) &= \Hom^{\le 0}(\aA_Y,f_{k+1}( \oO_C(k+1)[n]))\\
    &= \Hom^{\le 0}(\aA_Y,f_*\oO_C[n])\\
    &= \Hom^{\leq 0}(\aA_Y,\oO_o[n]).
  \end{align*}
  When $n <0$, we have $\Hom^{\le 0}(\aA_Y,\oO_o[n]) = 0$.
  Next, we show that the condition (2) in Theorem \ref{thm:glued} holds.
  Choose a real number $a$ with $0 < a < \frac{\Imm \lambda}{\pi} + n \leq 1$. 
  In this case, $\rho_{k+1}^L (Q_{\lambda}(a,a+1]) = \langle \oO_C(k+1)[n]\rangle$.
  Let $\pP_Y$ be a slicing corresponding to $\sigma_Y$.
  Since $\pP_Y(a,a+1]$ contains any $\oO_y$ for any closed point $y \in Y$,
  we have $\Hom^{<0}(\pP_Y(a,a+1],\oO_y) = 0$.
  If $n < 0$, $\Hom^{\leq 0}(\pP_Y(a,a+1],\oO_C(k+1)[n]) = 0$.
  Therefore, there exists a glued stability condition by Theorem \ref{thm:glued}.
  By Proposition \ref{prop:support}, the glued pre-stability condition $\sigma_Y *_{\rR_{k+1}} \tau_{\lambda}$ satisfies the support property.
\end{proof}

\begin{defi}\label{def:region_prime}
  Let $U(\rR_{k+1})$ be the set of glued stability conditions $\sigma_{Y} *_{\rR_{k+1}} \tau_{\lambda}$ for $\sigma_Y \in U(Y)$ and $\Imm \lambda > \pi$.
  For an interval $I$ in $\bbR$, we define $U_I(\rR_{k+1}) \subset \Stab_n(X)$ to be 
  \[
  \{ \sigma_Y *_{\rR_{k+1}} \tau_{\lambda} \in U(\rR_{k+1}) \colon \Imm \lambda \in \pi I\}.
  \]
\end{defi}
The set $U(\rR_{k+1}) \subset \Stab_n(X)$ is called a glued region with respect to $\rR_{k+1}$.

\begin{prop}\label{prop:stableobject}
Set $I = (1,2)$.
Let $\sigma = (Z,\aA)$ be a glued stability condition in $U_I(\rR_{k+1})$.
Then, $\oO_C(k)[1]$ is a $\sigma$-stable object in $\aA$.
  \end{prop}
\begin{proof}
  Let $\sigma = \sigma_Y *_{\rR_{k+1}} \tau_{\lambda}$ for $\sigma_Y \in U(Y)$ and $\Imm \lambda > \pi$.
  First, we show that $\oO_C(k)[1]$ lies in the heart $\aA$.
  Since it follows from the Serre duality that $\Hom^*_X(\oO_C(k+1),\oO_C(k)[2]) =\bbC $, $\rho_{k+1}(\oO_C(k)[1]) = \oO_C(k+1)[-1]$.
  There is a triangle in $\dD^b(X)$ 
  \[
    \oO_C(k+1)[-1] \rightarrow \oO_C(k)[1] \rightarrow f_{k+1}^L \oO_o \overset{+}{\rightarrow} 
  \]
  since we have $f_{k+1}^Lf_{k+2}(\oO_C(k)[1]) = f_{k+1}^L\oO_o$.
  Thus, $\oO_C(k)[1] \in \aA$.
  Next, we show that $\oO_C(k)[1]$ is $\sigma$-stable.
  Let $E$ be a subobject of $\oO_C(k)[1]$ in $\aA$.
  By Lemma \ref{lem:texact}, there exist two injections 
  \[
    f_{k+2}(E) \hookrightarrow \oO_o, \ \rho_{k+1}(E) \hookrightarrow \oO_C(k+1)[-1].
  \]
  Since both objects $\oO_o$ and $\oO_C(k+1)[-1]$ are simple objects,
  the object $E$ is isomorphic to $\oO_C(k+1)[-1]$ , $f_{k+1}^L(\oO_o)$ or an extension of them.
  We claim that $E$ is isomorphic to $\oO_C(k+1)[-1]$ or $\oO_C(k)[1]$.
  Indeed, since $\Hom^*(f_{k+1}^L \oO_o,\oO_C(k+1)[-1]) = 0$, $E$ is not isomorphic to $f_{k+1}^L(\oO_o)$.
  Therefore, $\rho_{k+1}(E)$ is non-zero.
  There is an exact sequence
  \[
  0 \rightarrow \oO_C(k+1)[-1] \rightarrow E \rightarrow f_{k+1}^Lf_{k+2}(E) \rightarrow 0.
  \]
  If $f_{k+1}^Lf_{k+2}(E)$ is non-zero, then $E$ is isomorphic to $\oO_C(k)[1]$.
  Therefore, we assume that $f_{k+1}^Lf_{k+2}(E)$ is zero.
  By the construction, we have that 
  \[
    \lambda -\pi = \arg(Z({\oO_C(k+1)[-1]}))  <\arg(Z({\oO_C(k+1)[-1]})-1)= \arg(Z({\oO_C(k)[1]})).
  \]
  Therefore, $\oO_C(k)[1]$ is $\sigma$-stable.
\end{proof}

The following lemma is needed to compare the glued stability conditions with the $(C_k)$-type stability conditions.
\begin{lem}\label{lem:pushforword}
  Let $\lambda = a + i\pi$ be a complex number with a real number $a$,
   and $\sigma_Y = (Z_Y,\aA_Y)$ be a normalized geometric stability condition on $Y$.
   Fix a positive real number $\varepsilon_0$.
  Consider a family of stability conditions $\sigma_{\varepsilon} = \sigma_Y *_{\rR_{k+1}} \tau_{\lambda+i\varepsilon}$ for $\varepsilon \in (0,\varepsilon_0]$ in \cref{prop:Rprimeglued}.
  Assume an object $E$ is $\sigma_{\varepsilon}$-stable for all $\varepsilon \in (0,\varepsilon_0]$.
  If $E$ lies in $\aA_{\varepsilon}$ for all $\varepsilon \in (0,\varepsilon_0]$, then  $f_{k+1}E$ lies in $\aA_Y$ or $E$ is isomorphic to $\oO_C[-1]$.
  Here, $\aA_{\varepsilon}$ is the heart of $\sigma_{\varepsilon}$.
\end{lem}
\begin{proof}
  By tensoring $\oO(kC)$, it is enough to show the statement for $k=-1$.
  Let $\aA$ be the heart $\aA_{\varepsilon}$ for a fixed real number $\varepsilon \in (0,\varepsilon]$.
  Let $E$ be a $\sigma_{\varepsilon}$-stable object in $\aA_{\varepsilon}$ for all $\varepsilon \in (0,\varepsilon_0]$.
  There is an exact sequence in $\aA$:
  \[
  0 \rightarrow F \rightarrow E \rightarrow f^* f_{1}E \rightarrow 0.
  \]Here, $F$ is a direct sum of $\oO_C[-1]$.
  Since $f_*F[1]$ is a skyscraper sheaf,
  applying the functor $f_*$ gives an exact triangle in $\dD^b(Y)$:
  \[
  (\oO_o[-1])^n \rightarrow f_*E \rightarrow f_*f^*f_1(E) \rightarrow (\oO_o)^n.
  \]
  Here, $n$ is a non-negative integer, and $o$ is a blowup locus of $f$.
  Since $f_*f^*f_1(E) \cong f_1(E)$ lies in $\aA_Y$, 
  applying the cohomology functor $\hH^*_{\aA_Y}(-)$ gives an exact sequence
  \[
  0 \rightarrow \hH^0_{\aA_Y}(f_*E) \rightarrow f_1(E) \rightarrow \oO_o^n \rightarrow \hH^1_{\aA_Y}(f_*E) \rightarrow 0.
  \]
  Thus, there exists a surjection $\oO_o^n \rightarrow \hH^1_{\aA_Y}(f_*E)$.
  Since $\oO_o$ is a simple object in $\aA_Y$, 
  $\hH^1_{\aA_Y}(f_*E)$ is isomorphic to $\oO_o^k$ for some $k$.

  Suppose that $\hH^1_{\aA_Y}(f_*E)$ is non-zero.
  Define  $K$ to be the cone of $f^*f_*E \rightarrow E$.
  Suppose that $K$ is zero.
  Then, $E$ is isomorphic to $f^*f_*E$.
  Since $f^*$ is t-exact and fully faithful, $\hH^1_{\aA_Y}(f_*E)$ is zero.
  Therefore, $K$ is non-zero. 

  Next, we consider the case where $K$ is non-zero.
  The object $K$ lies in $\langle \oO_C(-1)\rangle$.
  Since the cohomology object $\hH_{\aA_Y}^i(f_*E)$ vanishes for $i \not = 0,1$,
  it follows from t-exactness that $\hH^i(f^*f_*E) = 0$ for $i \not = 0,1$.
  By applying the cohomology functor $\hH_{\aA}^*$ to the triangle $f^*f_*E \rightarrow E \rightarrow K$,
  $\hH_{\aA_Y}^i(K)$ vanishes for $i \not = -1,0$.
  Let $h\colon E \rightarrow \hH_{\aA}^0(K)$ be a map such that there is an exact sequence
  \[
  0 \rightarrow \hH_{\aA}^{-1}(K)  \rightarrow \hH_{\aA}^{0}(f^*f_*E) \rightarrow E \overset{h}{\rightarrow} 
  \hH^0_{\aA}(K) \rightarrow \hH^1_{\aA}(f^*f_*E) \rightarrow 0.
  \]
  Since $\hH_{\aA}^1(f^*f_*E)$ is isomorphic to a direct sum of  $f^*\oO_o$, the object $\hH_{\aA}^0(K)$ is not zero.
  The object $\hH_{\aA}^0(K)$ is isomorphic to $\oO_C(-1)[1]$, and the map $\hH^0_{\aA}(K) \rightarrow \hH^1_{\aA}(f^*f_*E)$ is not an isomorphism.
  Thus, the map $h$ is non-zero. 
  There is an exact sequence in ${\aA}$:
  \[
  0 \rightarrow \oO_C[-1] \rightarrow \oO_C(-1)[1] \rightarrow f^*\oO_o \rightarrow 0.
  \]
  Since $\hH^0_{\aA}(K)$ is isomorphic to a direct sum of  $\oO_C[-1]$, the map $h$ induces a map $i\colon E \rightarrow \oO_C[-1]$.
  Since $E$ is $\sigma_{\varepsilon}$-stable for all $\varepsilon \in (0,\varepsilon_0]$, the map $i$ is an isomorphism.
  Then, the statement holds.
\end{proof}

\subsubsection{Stability conditions with hearts of type $(\lL_k)$.}
We discuss the $\lL_k$-glued stability conditions.
The recollement $\lL_k$ is defined by
\[
  \lL_k \colon
      \begin{tikzcd}
        \langle \oO_C(k)\rangle \arrow[r,"j" description] & \dD^b(X) \arrow[l,bend right=40,"\rho_{k+2}" description] \arrow[r,"f_{k+1}"] & \dD^b(Y)  \arrow[l,bend right=40,"f^L_{k+1}" description].
        \end{tikzcd}
\]
\begin{prop}
  Let $\lambda$ be a complex number with $\Imm \lambda \leq -\pi$,
  and $\sigma_Y= (Z_Y,\aA_Y)$ be a normalized geometric stability condition on $Y$.
  Then, $\tau_{\lambda} *_{\lL_k} \sigma_Y$ defines a stability condition.
\end{prop}
\begin{proof}
  We show that the condition (1) and (2) in Theorem \ref{thm:glued} hold.
  Since $\Imm \lambda \leq -\pi$,  $\qQ_{\lambda}(0,1] = \langle \oO_C(k)[n]\rangle$ for an integer $n > 1$.
  Since $\aA_Y$ contains all skyscraper sheaves $\oO_y$, we have 
  \begin{align*}
    \Hom^{\le 0}(\qQ_{\lambda}(0,1],f_{k+1}^L\aA_Y) &= \Hom^{\le 0}(f_{k+2}(\oO_C(k)[n]),\aA_Y)\\
    &= \Hom^{\le 0}(f_*(\oO_C(-2)[n]),\aA_Y)\\
    &= \Hom^{\le 0}(\oO_o[n-1],\aA_Y) = 0.
  \end{align*}
  Next, we show the condition (2) in Theorem \ref{thm:glued}.
  The same proof in Proposition \ref{prop:Rprimeglued} works for this case.
\end{proof}

\begin{defi}\label{defi:region}
  Let $U(\lL_{k})$ be a subset of $\Stab_n(X)$ defined to be 
\[
U(\lL_{k}) = \left\{\tau_{\lambda} *_{\lL_{-1}} \sigma_Y \colon \sigma_Y \in U(Y) \text{ and }\Imm \lambda \leq -\pi  \right\}.
\]
\end{defi}

\begin{table}[h]
  \small
  \centering
  \caption{Comparison of stability conditions}
  \label{tab:hogehoge}
  \begin{tabular}{c|ccc}
       & Notation & Central charge & Heart \\ \hline
      Geometric  & $\sigma_{B,\omega_X}$ & $Z_{B,\omega_X}$ & $\aA_{B,\omega_X}$ \\
      Type $(C_k)$ &  $\sigma_{B,f^* \omega_Y}$ ($k - \frac{1}{2} < B.C < k + \frac{1}{2} $)  & $Z_{B,f^*\omega_Y}$ & $\aA_{B,f^*\omega_Y} = \langle \oO_C(k)[1]\rangle *_{\rR_k} \aA_{f_*B,\omega}$ \\
      $(\rR_{k+1})$-glued  & $\sigma_{B,\omega_Y} *_{\rR_{k+1}} \tau_{\lambda}$ ($\Imm \lambda >\pi$) & $Z_{B,\omega_Y}\circ f_{k+2} \oplus W_{\lambda}\circ\rho_{k+1} $
      &  $\aA_{B,\omega_Y}*_{\rR_{k+1}} \langle \oO_C(k)[n]\rangle$ \\
      $(\lL_k)$-glued & $\tau_{\lambda}*_{\lL_k}\sigma_{B,\omega_Y}$ ($\Imm \lambda \leq -\pi) $ & $W_{\lambda}\circ \rho_{k+2} \oplus Z_{B,\omega_Y}\circ f_{k+2}$
      & $\langle \oO_C(k-1)[n]\rangle *_{\lL_k} \aA_{B,\omega_Y} $ \\
  \end{tabular}
\end{table}

\subsection{Gluing regions of $\Stab(X)$}
In this section, we discuss the boundary of a gluing region with respect to the recollement $\rR_{k+1}$.
By considering the action of $\oO(kC)$ on $\Stab(X)$, it is enough to consider the case when $k=-1$.

Let $\sigma_Y = (Z_Y,\aA_Y)$ be a stability condition on $Y$.
We have the map $(\bbL f^*)^{\vee} \colon \Hom(H^*(X;\bbC),\bbC) \rightarrow \Hom(H^*(Y;\bbC),\bbC)$
 defined by the dual map of $[\bbL f^*] \colon H^*(Y;\bbC) \rightarrow H^*(X;\bbC) $.
We consider the subspace $\Stab(X)_{Z_Y}$ of $\Stab(X)$ defined by the following Cartesian diagram:
\[
\begin{tikzcd}
\Stab(X)_{Z_Y} \arrow[r] \arrow[d] & \Stab(X) \arrow[d,"(\bbL f^*)^{\vee} \circ \pi"] \\
\{Z_Y\} \arrow[r,hook] & \Hom(H^*(Y;\bbC),\bbC).
\end{tikzcd}
\]
We discuss walls of $\Stab(X)_{Z_Y}$ instead of $\Stab(X)$.
We assume that $\sigma_Y$ is a normalized geometric stability condition with $Z_Y(\oO_y)=-1$ for any closed point $y \in Y$.
By \Cref{sec:gluingregion}, the fiber $\Stab(X)_{Z_Y}$ is non-empty and a one-dimensional complex submanifold of $\Stab(X)$.

We determine the boundaries obtained by considering the limit of $\Imm \lambda \rightarrow \pi$.

\begin{thm}\label{thm:gluingregion}
Let $\sigma_Y = (Z_Y,\aA_Y)$ be a normalized geometric stability condition on $Y$, and let $U(\rR_{0})_{\sigma_Y}$ be a set defined by 
\[
  U(\rR_{0})_{\sigma_Y} = \left\{ \sigma \in U(\rR_{0}) \colon \sigma \text{ is glued from }\tau_{\lambda} \text{ and } \sigma_Y \right\}.
\]
Then, the boundary of $U(\rR_{0})_{\sigma_Y}$ in $\Stab(X)_{Z_Y}$ decomposes into the disjoint union of two components $W_0$ and $W_1$.
Here, the sets are defined by
\begin{align*}
  &W_0 = \{\tau \in \partial U(\rR_{0}) \colon \oO_C(-1) \text{ is semistable of phase } 0 \},\\
  &W_{-1} = \{\tau \in \partial U(\rR_{0}) \colon \oO_C(-1) \text{ is semistable of phase } -1 \}.
 \end{align*}
 Moreover, the following two properties hold.
 \begin{enumerate}
  \item For any $\tau \in W_0$, if $\oO_C$ and $\oO_C(-1)[1]$ are $\tau$-stable, then $\tau$ lies in the boundary of geometric chamber $U(X)$.
  \item The set $W_{-1}$ is contained in $U(\lL_{-1})$.
 \end{enumerate}
\end{thm}

\begin{proof}
  The boundary $\partial U(\rR_{0})_{\sigma_Y}$ is contained in the set
  \[
    \{\tau \in \Stab(X) \colon \oO_C \text{ is semistable of phase } 1  \},
  \]
  since the object $\oO_C(-1)$ is $\sigma$-stable for $\sigma \in U(\rR_{0})$.
By Proposition \ref{prop:stableobject}, $\oO_C(-1)[1]$ is $\tau$-semistable of phase $\phi_{\tau}(\oO_C(-1)[1])$ for $\tau \in \partial U(\lL_{-1})$.
 For any stability condition $\tau \in \partial U(\lL_{-1})$ since $Z(\oO_C(-1)[1]) \in \bbR$, 
  the phase $\phi_{\tau}(\oO_C(-1)[1])$ is equal to $0$ or $1$.
It follows that $\partial U(\lL_{-1})$ is contained in the union of $W_0$ and $W_{-1}$.

Next, we show the first property.
Let $\tau=(Z_{\tau},\aA_{\tau}) \in W_0$ be a stability condition such that $\oO_C$ and $\oO_C(-1)[1]$ are $\tau$-stable.
By the definition of $W_0$, the objects $\oO_C$ and $\oO_C(-1)[1]$ is $\tau$-stable of phase $1$.
We will show the following claims:
\begin{enumerate}
  \item[(i)] For any $x \in X \backslash C$, the object $\oO_x$ is $\tau$-stable of phase $1$.
  \item[(ii)] There exists a stability condition $\sigma$ arbitrarily close to $\tau$ such that $\oO_x$ are $\sigma$-stable for any $x \in C$.
\end{enumerate}
We show the claim (i).
Suppose that there is a point $x \in X \backslash C$ and a $\tau$-stable object of phase one
with an injection $E \hookrightarrow \oO_x$.
Take a stability condition $\rho = (Z_{\rho},\aA_{\rho}) \in \partial U(\rR_{0})_{\sigma_Y}$  arbitrarily close to $\tau$.
Since $\oO_x$ is stable of phase one, the object $E$ lies in $\aA_{\rho}$.
By the construction of $\aA_{\rho}$, there is  an exact sequence in $\aA_{\rho}$:
\[
0 \rightarrow \oO_C^k[-1] \rightarrow E \rightarrow f^*f_1(E) \rightarrow 0.
\]
Since $x$ is not contained in $C$, $\Hom(\oO_C,\oO_x)=0$.
Therefore, there is a non-zero morphism $f^*f_1(E) \rightarrow \oO_x$.
However, since $E$ is $\tau$-stable of phase one,
$Z_{\rho}(f^*f_1(E)) = Z_Y(f_1(E)) = Z_{\tau}(f^*f_1(E))$ is contained in $\bbR$.
This is a contradiction to the fact that $\Hom(f^*f_1(E),\oO_x) \not=0$.

We next show the claim (ii).
Fix a complex number $\lambda$ with $\Imm \lambda >0$.
Consider the stability condition $\tau_{\lambda} = (Z_{\lambda},\aA_{\lambda})$ defined by
\begin{align*}
  &Z_{\lambda}(f^*E) = Z_Y(E),\\
  &Z_{\lambda}(\oO_C(-1)) = \lambda.
\end{align*}
By \Cref{thm:deformation}, there is a stability condition $\tau_{\lambda}$ for $\lambda$ enough small.
Suppose that there is a point $x \in C$ and a $\tau_{\lambda}$-stable object $E$ of phase one
with an injection $E \hookrightarrow \oO_x$.
Since walls with respect to skyscraper sheaves are locally finite,
 we can assume that $E$ is a $\tau$-semistable object of phase one.
 There is an exact triangle
 \[
 f^*f_*E \rightarrow E \rightarrow F \rightarrow f^*f_*E[1].
 \]
Here, the object $F$ is in $\langle \oO_C(-1)[1] \rangle$.
By construction, $Z_{\lambda}(E) = Z_Y(f_*(E)) + \lambda \chi(\oO_C(-1),F)$.
Here, $\chi(\oO_C(-1),F)$ is a Eular pairing of $\oO_C(-1)$ and $F$.
Since $E$ is a $\tau_{\lambda}$-stable object $E$ of phase one,
it holds that $\chi(\oO_C(-1),F) = 0$.

Consider the case when $F$ is non-zero.
If $\oO_C(-1)[n]$ is a direct summand of $F$ for some $n$,  $1\leq n$.
Indeed, the objects $E$ and $\oO_C(-1)[1]$ lies in $\aA_{\tau}$.
Thus, $\Hom(E,\oO_C(-1)[n])=0$ for $n<1$.

Serre duality implies that $\Hom(E,\oO_C(-1)[n])=0$ for $n>3$.
Therefore, it follows from $\chi(\oO_C(-1),F) = 0$ that $\oO_C(-1)[2]$ is a subobject of $F$.
Since $\Hom(\oO_C,\oO_C(-1)[2])\not=0$, there is a map $\oO_C \rightarrow E$.
Since $\Hom(\oO_C,\oO_x)$ is one-dimensional, .the following diagram is commutative:
\[
\begin{tikzcd}
\oO_C \arrow[r] \arrow[d] & E \arrow[l d] \\
\oO_x  & .
\end{tikzcd}
\]
Since there is a exact sequence $0 \rightarrow \oO_C(-1) \rightarrow \oO_C \rightarrow \oO_x \rightarrow 0$ in $\aA_{\lambda}$, the map $\oO_C \rightarrow \oO_x$ is surjective.
The map $E \rightarrow \oO_x$ is injective and surjective in $\aA_{\lambda}$.
Then, the object $E$ is isomorphic to $\oO_x$.

Consider the case of $F=0$.
As $E \cong f^*f_*E$,
$\Hom(E,f^*f_*\oO_x) \not=0$.
There is an exact triangle:
\[
\oO_C(-1)[1] \rightarrow f^*f_*\oO_x \rightarrow \oO_C \overset{+}{\rightarrow} .
\]
Since $\Hom^*(f^*f_*E,\oO_C(-1))=0$, there is a non-zero map $E \rightarrow \oO_C$.
Take $\lambda$ enough small so that $\oO_C$ is $\tau_{\lambda}$-stable.
It contradicts the fact that $E$ is $\tau_{\lambda}$-stable of phase one.
Therefore, we have shown that $\tau_{\lambda}$ is a geometric stability condition.

Finally, we show that $W_{-1}$ lies in $U(\lL_{-1})$.
Let $\tau=(Z,\pP) \in W_{-1}$.
Then, the objects $\oO_C$ and $\oO_C(-1)[2]$ lie in $\pP(1)$ and $f^*\aA_Y \subset \pP(0,1]$.
Consider $\bB = \langle \oO_C(-1)[2] \rangle *_{\lL_{-1}} \aA_Y$.
Since any object $E \in \bB$ can be obtained from a extension of $\oO_C(-1)[2]$ and an object in $\aA_Y$, we have $\pP(0,1] \subset \bB$.
Therefore, $\tau$ lies in $U(\lL_{-1})$.
\end{proof}

\begin{cor}
  Two gluing regions $U(\rR_{0})_{\sigma_Y}$ and $U(\lL_{-1})_{\sigma_Y}$ lie in the connected component $\Stab_n^{\dagger}(X)$.
  Here, $\Stab_n^{\dagger}(X)$ is the connected component of $\Stab_n(X)$ containing $U(X)$.
\end{cor}

\section{Noncommutative minimal model program for a blowup at a point} \label{sec:ncmmp}

\subsection{Quasi-convergent paths on $\Stab(X)$}
We discuss Proposal \ref{proposal} introduced in \cite{HL:23} in the case of a blowup surface.
Let $f \colon X \rightarrow Y$ be a blowing up at a point of a smooth projective surface $Y$.
\begin{defi}[{\cite{HLJR}}]
  Let $\sigma$ be a stability condition on $X$ and $E$ be an object in $\dD^b(X)$.
  We define the phase of $E$ with respect to $\sigma$ to be
  \[
  \phi_{\sigma}(E) = \frac{1}{m_{\sigma}(E)} \sum_{i} \phi_{\sigma}(F_i) m_{\sigma}(F_i).
  \]
  Here, $F_i$ runs all  semistable objects appearing in the HN filtration of $E$ with respect to $\sigma$.
  We also define the logarithm of $E$ with respect to $\sigma$ to be
  \[
  \ell_{\sigma}(E) = \log m_{\sigma}(E) + i \pi \phi_{\sigma}(E).
  \] 
  For a pair of objects $E$ and $F$, we define the logarithm of the ratio to be
  \[
  \ell_{\sigma}(E/F) = \ell_{\sigma}(E) - \ell_{\sigma}(F).
  \]
\end{defi}
\begin{defi}
  We call a continuous map $\sigma_{\bullet} \colon [T_0,\infty) \rightarrow \Stab(X)$ a path of stability conditions.
  We say that $E$ is $\sigma_{\bullet}$-limit semistable if $ \lim_{t\rightarrow \infty}(\phi^+_{\sigma_t}(E) - \phi^- _{\sigma_t}(E)) = 0$.
\end{defi}
For simplicity of notation, we write $\phi_t(E)$ and $\ell_t(E)$ instead of $\phi_{\sigma_t}(E)$ and $\ell_{\sigma_t}(E)$.
\begin{defi}[{\cite[Definition 2.8.]{HLJR}}] \label{def:qcpath}
  A path $\sigma_{\bullet} \colon [T_0,\infty)\rightarrow  \Stab(X)$ is called a quasi-convergent path if the following conditions hold.
  \begin{enumerate}
    \item For any object $E \in \dD^b(X)$, there exists a filtration 
    \[ 0 \rightarrow E_1 \rightarrow \cdots \rightarrow E_n = E \]
    such that the cone $G_i = \Cone(E_{i-1} \rightarrow E_i)$ is $\sigma_{\bullet}$-limit semistable, and for all $i$ 
    \[ \liminf_{t\rightarrow \infty} \left\{ \phi_t(G_{i-1}) - \phi_t(G_i) \right\}>0 . \]
    \item For any pair of limit semistable objects $E$ and $F$, the following limit 
    \[ \lim_{t \rightarrow \infty}\frac{\ell_t (E/ F)}{ 1+ |\ell_t (E/ F)|} \]
    exists.
  \end{enumerate}
\end{defi}
We consider an equivalence relation to define subcategories of $\dD^b(X)$.
\begin{defi}[{\cite[Definition 2.16.]{HLJR}}]
  Let $E$ and $F$ be objects in $\dD^b(X)$.
  We define the following relations.
\begin{enumerate}
  \item $ E \prec F$ if $ \lim_{t \rightarrow \infty} (\phi_t(F) -\phi_t(E)) = \infty$.
  \item $E \sim^i F$ if $\limsup_{t \rightarrow \infty} |\phi_t(E) -\phi_t(F)| < \infty$.
  \item $E \sim F$ if $\lim_{t \rightarrow \infty} \ell_t(E/F)$ exists.
\end{enumerate}
\end{defi}

\begin{thm}[{\cite[Proposition 2.20.,Theorem 2.37.]{HLJR}}]
  Let $\sigma_{\bullet} \colon [T_0,\infty) \rightarrow \Stab(X)$ be a quasi-convergent path.
  Assume that the equivalence relations $\sim = \sim^i$ are the same.
  Then there exist finite number of  objects $E_1 ,\dots, E_n$ such that 
\begin{enumerate}
  \item $E_1 \prec E_2 \prec \cdots \prec E_n$, and 
  \item $\langle \cC^{E_1},\dots,\cC^{E_n} \rangle$ defines a semiorthogonal decomposition of $\dD^b(X)$.
\end{enumerate}
Moreover, the path $\sigma_{\bullet}$ induces a stability condition on each subcategory $\cC^{E_i}$.
Here, $\cC^E$ is defined to be
\[
\cC^E = \left\{ F \in \dD^b(X) \colon E \sim F \right\}.
\]
\end{thm}
From the theorem above, we can expect a similar conjecture for Gamma conjecture II.
\begin{proposal}\label{proposal}
  There is a path $\sigma_{\bullet} = (Z_{\bullet}, \aA_{\bullet}) \colon [T_0,\infty) \rightarrow \Stab(X)$ with the following conditions.
  \begin{enumerate}
    \item $\sigma_{\bullet}$ is a quasi-convergent path.
    \item The family of central charges $Z_{\bullet}$ satisfies the differential equation (\ref{eq:qde}).
  \end{enumerate}
\end{proposal}

We define the following central charge for a continuous function $g \colon [T_0,\infty) \rightarrow \bbC$ to be
\begin{equation}\label{eq:pathCC}
	\zZ(t,Z_0,F) = g(t) c_1(F).C  + Z_0(F).
\end{equation}
Here, $Z_0$ lies in  $\Hom(H^*(X;\bbC),\bbC)$.
From now on, we make the assumptions.:
\begin{assum}\label{assumption}
\begin{enumerate}
  \item $Z_t(\oO_C(-1)) \not = 0 $ and $Z_t(\oO_C) \not = 0 $ for all $t \in [T_0,\infty)$. \label{assum:1}.
  \item $\arg Z_t(\oO_C(-1))$ is a monotonically decreasing function and diverges to $-\infty$.  \label{assum:2}
  \item The absolute value $|Z_t(\oO_C(-1))|$ diverges to $\infty$. \label{assum:3}
  \item The following limit 
  \[
    \frac{\ell_t (\oO_C(-1))}{ 1+ |\ell_t (\oO_C(-1))|}
  \]exists.\label{assum:4}
\end{enumerate}
\end{assum}
Generally, the argument of a complex number is a multi-valued function, or it is only well-defined on an appropriate open set.
However, since $Z_t(\oO_C(-1))$ does not vanish on $\bbR_{\ge 0}$, 
if we take the branch of the argument function, we can define $\arg Z_t(\oO_C(-1))$ as a continuous function on $[T_0,\infty)$.

First, we discuss the case when the initial value lies in the gluing region $U(\rR_{0})$.
\begin{lem}\label{lem:start rR prime -1}
  Let $\sigma_0$ be a stability condition in $U(\rR_{0})$.
  If \Cref{assumption} holds, the path of central charges $Z_{\bullet}$ defined by the equation (\ref{eq:pathCC}) lifts to the path of stability conditions $\sigma_{\bullet}$ with the initial value $\sigma_0$.
  Moreover, if $\oO_C$ is $\sigma_t$-stable of phase $\phi_t(\oO_C) >1$ for enough large $t \in [T_0,\infty)$, then $\sigma_t$ lies in $U(\rR_{0})$.
  If $\oO_C(-1)$ is $\sigma_t$-stable of phase $\phi_t(\oO_C(-1)) <-1$ for enough large $t \in [T_0,\infty)$, then $\sigma_t$ lies in $U(\lL_{-1})$.
\end{lem}

\begin{proof}
Consider the argument $\arg(Z_t(\oO_C)) \colon \bbR_{\geq T_0} \rightarrow \bbR$.
Here, we choose the branch of the argument function with $\arg(Z_{T_0}(\oO_C)) = \phi_{T_0}(\oO_C)$.
By \Cref{assumption} (\ref{assum:1}), this function is well-defined.
By \Cref{assumption} (\ref{assum:2}-\ref{assum:3}), the function $\arg(Z_t(\oO_C))$ diverges to $-\infty$.
Let $T_1$ be a real number such that $\arg(Z_{T_1}(\oO_C)) = 1$.
First we consider the case when $T_0 \leq t < T_1$. 
  Take $\lambda_t$ as a complex number with $\Imm \lambda_t = \arg(Z_t(\oO_C))$, and $e^{\lambda_t} = Z_t (\oO_C)$.
  Then, there exists a stability condition $\sigma_Y$ on $Y$ 
  such that $\sigma_0$ is glued from $\sigma_Y$ and $\tau_{\lambda_0}$ with respect to the recollement $\rR_{0}$.
  The family of stability conditions $\sigma_t$ for $T_0 \leq t < T_1$ can be defined by glued stability conditions  $\sigma_Y *_{\rR_{0}} \tau_{\lambda_t}$.

  We assume that there exists a real number $T_1 > T_0$ satisfying $\Imm Z_{T_1}(\oO_C(-1))=0$.
  Let $\nu$ be a complex number satisfying $\Imm \nu = -\pi$, and $Z_{T_1}(\oO_C(-1)) = e^\nu$.
  We define $\tau$ to be $\tau_{\nu} *_{\lL_{-1}} \sigma_Y$.
  By \Cref{thm:deformation}, there is  a path of stability conditions $\tau_t = (W_t,\bB_t) \in U(\lL_{-1})$ for $t \in (T_1-\varepsilon, T_1+ \varepsilon)$ 
  such that its family of central charges $W_t$ is defined by the equation (\ref{eq:cc}) and $\tau_{T_1} = \tau$.
  \Cref{thm:gluingregion} shows that $\sigma_t = \tau_t$ for all $t \in (T_1-\varepsilon, T_1)$.
  It follows from \Cref{assumption} (\ref{assum:1}) that $\sigma_t$ is defined to be glued stability conditions with respect to $\lL_{-1}$ for $t > T_1$.
  We conclude that the path of central charges $Z_t$ for $t \in \bbR_{\geq 0}$ lifts to the path of stability conditions $\sigma_t$.
  By construction of the path, we obtain the second and third statements.
\end{proof}

\begin{prop}\label{prop:qcpath}
  We assume the path of central charges $Z_{\bullet} \colon [T_0,\infty) \rightarrow \Stab(X)$ 
  defined by the equation (\ref{eq:pathCC}) lifts to the path of stability conditions $\sigma_{\bullet}$.
  We also assume that $\sigma_{T_1} \in U(\lL_{-1})$ for some $T_1$.
  If $\lim_{t \rightarrow 0} \arg Z_t(\oO_C(-1)) = -\infty$ , then the path $\sigma_{\bullet}$ is quasi-convergent.
\end{prop}
\begin{proof}
  We need to show the conditions in Definition \ref{def:qcpath}.
  Since   $\sigma_t \in U(\lL_{-1})$ for $t > T_1$,
  every limit semistable object lies in $\langle \oO_C(-1) \rangle$ or $f^*\aA_Y$.
  Therefore, all limit semistable objects are $\sigma_t$-semistable for $t \gg 0$.

  For any object $E \in \dD^b(X)$, there is a distinguished triangle
  \[
    f^*f_* E \rightarrow E \rightarrow V_E \otimes \oO_C(-1) \overset{+}{\rightarrow} 
  \]
  for an object $V_E$ in the derived category of vector spaces $\dD^b(\Vect)$.
  Since the HN filtration of each object in  $\langle \oO_C(-1) \rangle$ or $f^*\aA_Y$ is independent from $t>T_1$,
  the HN filtration of $E$ stabilizes for $t \gg 0$.
  Then, the condition (1) holds.
  We show the condition (2).
  Let $E$ and $F$ be semistable objects in $\dD^b(X)$.
  We can assume that $E$ lies in $\langle \oO_C(-1) \rangle$.
  In this assumption, since we have $E = \bigoplus_i(\oO_C(-1)[n_i])$ for $n_i \in \bbZ$,
  $\lim_{t \rightarrow \infty}\ell_t(E) - \ell_t(\oO_C(-1))$ exists in $\bbC$.
  If $F$ lies in $\langle \oO_C(-1) \rangle$,
  then $\ell_t(E/F)$ exists.
  Assume $F$ lies in $f^* \aA_Y$.
  \begin{align*}
    \left| \frac{\ell_t(E / F)}{1 + |\ell_t(E / F)|} \right| &\le \frac{| \ell_t(E)|+ |\ell_t(F)|}{ 1+ |\ell_t(E)|-|\ell_t(F)|}\\
    &=   \frac{\frac{|\ell_t(E)|}{|\ell_t(E)|}+ \frac{|\ell_t(F)|}{|\ell_t(E)|}}{ 1+ \frac{1}{|\ell_t(E)|}-\frac{|\ell_t(F)|}{|\ell_t(E)|}}
  \end{align*}
  Since the limit $\lim_{t\rightarrow \infty}\frac{\ell_t(E)}{|\ell_t(E)|}$ exists by \Cref{assumption} (4), the  condition (2) holds.
\end{proof}

\begin{prop}
  Under \Cref{assumption}, both equivalence relations $\sim = \sim^i$ are the same.
  This quasi-convergent path $\sigma_{\bullet}$ induces a semiorthogonal decomposition $\langle \oO_C(-1),\dD^b(Y) \rangle$.
\end{prop}
\begin{proof}
  Let $E$ and $F$ be semistable objects in $\dD^b(X)$.
  It is enough to show that $\lim_{t \rightarrow \infty}\ell_t(E / F)$ exists for semistable objects $E$ and $F$ with $ \lim_{t \rightarrow \infty}|\phi_t(E) - \phi_t(F)| < \infty $.
  As in the proof of Proposition \ref{prop:qcpath}, if a pair of semistable objects $E$ and $F$ satisfies $ \lim_{t \rightarrow \infty}|\phi_t(E) - \phi_t(F)| < \infty $,
  we have $E$ and $F$ lie in the same semiorthogonal component.
  Then, we conclude that $\ell_t(E / F) $ exists in $\bbC$.
  We have \[
    \dD_{\oO_C(-1)} = \langle \oO_C(-1) \rangle.
  \]
  Then, $\sigma_t$ induces $\langle \oO_C(-1),\dD^b(Y) \rangle$.
\end{proof}
Summarizing the above discussion and tensoring by a line bundle $\oO(kC)$, we have the following theorem.

\begin{thm}\label{thm:summarizing}
  Let $Z_t$ be a path of central charges defined by (\ref{eq:pathCC}) which lifts to the path of stability conditions $\sigma_t$ for all $t \in[T_0,\infty)$.
We assume the following conditions:
\begin{enumerate}
  \item the initial value $\sigma_{T_0}$ lies in the gluing region $U(\rR_{k+1})$ for $k \in \bbZ$, and 
  \item the absolute value $|Z_t(\oO_C(-1))|$ diverges to $\infty$,
  \item the following limit 
  \[
    \lim_{t \rightarrow \infty} \frac{\ell_t (\oO_C(k))}{ 1+ |\ell_t (\oO_C(k))|}
  \]exists.
\end{enumerate}
In addition, assume  that one of the following conditions holds:
\begin{enumerate}
  \item[(4)] the function $\arg Z_t(\oO_C(k))$ is monotonically decreasing with  $\lim_{t \rightarrow 0} \arg Z_t(\oO_C(k)) = -\infty$.
  \item[$(4^{\prime})$] the function $\arg Z_t(\oO_C(k+1))$ is monotonically increasing with  $\lim_{t \rightarrow 0} \arg Z_t(\oO_C(k+1)) = \infty.$
\end{enumerate}
Then, the path $\sigma_{\bullet}$ is quasi-convergent.
Moreover, in the case of $(4)$, the quasi-convergent path $\sigma_{\bullet}$ induces a semiorthogonal decomposition $\langle \oO_C(k),\dD^b(Y) \rangle$.
In the case of $(4^{\prime})$, the quasi-convergent path $\sigma_{\bullet}$ induces a semiorthogonal decomposition $\langle \dD^b(Y),\oO_C(k+1) \rangle$.
\end{thm}

\subsection{Noncommutative minimal model program for a blowup at a point}
We consider the noncommutative minimal model program for a blowup at a point.
Let $f \colon X \rightarrow Y$ be a blowing up at a point of a smooth projective surface $Y$.
We consider a family of paths of stability conditions whose central charges are the form of 
\begin{equation} \label{eq:cc}
  \zZ(s,\lambda,Z_0,t)(E) = w(s,\lambda,Z_0)(Ei(\lambda t)- Ei(\lambda T_0))\ch_1^{-sC}(E)C + Z_{0,sC}(E).
\end{equation}
for $s \in \bbR , \lambda \in \bbR \oplus i [-1,1]$ and $t \in \bbR_{\geq T_0}$.
Here, $Z_0$ is the initial central charge in $\Hom(H^*(X;\bbC),\bbC)$, and $Ei(\lambda)$ is the exponential integral.
$Z_{0,sC}$ is the central charge defined to be $Z_{0,sC}(E) = Z_0(e^{sC}\ch(E))$.
Since the exponential integral is defined in $\bbC \backslash \bbR_{\leq 0}$,
the central charge is defined for any $\lambda \in \bbR \oplus i [-1,1]$.
The continuous function  $w(s,\lambda,Z_0)$ needs to be modified for each theorem, and its definition will be provided later.

We consider the following question.
  \begin{question}\label{question}
    Does there exist an subset $\uU \subset \bbR \times (\bbR \oplus i[-1,1]) \times \bbR_{\geq T_0}$ , initial stability condition $\sigma_0 = (Z_0,\aA_0)$, and a continuous function $w(s,\lambda,Z_{0})$
    which satisfy the following conditions:
    \begin{enumerate}
      \item a continuous map  $\zZ(-,-,Z_0,-) \colon \uU \rightarrow \Stab(X)$ defined by the equation (\ref{eq:cc}) lifts to a map $\sigma_{\bullet,\bullet,\bullet}\colon \uU \rightarrow \Stab(X)$, and 
      \item there exists $(s,\lambda) $ 
      such that $(s,\lambda,t)$ lies in $\uU$ for all $t \geq T_0$ and $\sigma_{s,\lambda,\bullet}$ is a quasi-convergent path.
      \item some parameters $(s,\lambda)$ give a semiorthogonal decomposition, and different  parameters give a mutation-equivalent semiorthogonal decomposition.
    \end{enumerate}
    \end{question}
We use the following properties of the exponential integral.
\begin{lem}[{\cite[Section 5.]{ASM}, \cite[Chapter 6.]{DLMF}}]\label{lem:ei}
  Let $\lambda \in \bbC \backslash \bbR_{\leq 0}$.
  The function $\Ei(\lambda t)$ satisfies the following properties:
  \begin{enumerate}
    \item $\Ei(\lambda t)$ is a  differentiable function on $\bbR_{\geq 1}$, and
    \item  the asymptotic Expansion of $\Ei(\lambda t)$ is given by
    \[
    \Ei(\lambda t ) \sim \frac{e^{\lambda t}}{\lambda t} (1 + \frac{1!}{\lambda t} + \frac{2!}{\lambda t} + \cdots)\  t \rightarrow \infty,
    \]  
    \item there exist a real number $T_0$ such that the function $\arg(\Ei(\lambda t))$ is a monotonic function for enough large $t$. 
    If $\Imm \lambda >0$, then $\arg(\Ei(\lambda t))$ is a monotonically increasing function.
    If $\Imm \lambda <0$, then $\arg(\Ei(\lambda t))$ is a monotonically decreasing function.
    Here, $f(t) \sim g(t)$ means that $\lim_{t \rightarrow \infty} \frac{f(t)}{g(t)} =1$.
  \end{enumerate}
\end{lem}

First we consider the case when the initial value lies in $\partial U(X) \cap \partial U(\rR_{0})$.
Let us fix $\sigma_0 = (Z_0,\aA_0)$.
By \Cref{lem:ei}, for any $(s,\lambda)$ there exists a real number $T_0$ such that the following conditions hold:
\begin{enumerate}
  \item $\zZ(s,\lambda,Z_0,t)(\oO_C(-1)) \not = 0$ and $\zZ(s,\lambda,Z_0,t)(\oO_C) \not = 0$ for all $t\geq T_0$.
  \item  If $-1\leq\Imm \lambda <0$, $\arg \zZ(s,\lambda,Z_0,t)(\oO_C(-1))$ is  monotonically decreasing.
  \item If  $0<\Imm \lambda\leq 1$, $\arg \zZ(s,\lambda,Z_0,t)(\oO_C)$ is  monotonically increasing.
\end{enumerate}
We consider the continuous map $w(s,\lambda,Z_0) \colon  \bbR \oplus i [-1,1] \rightarrow \bbR$ for fixed $s$ and $\lambda$ with the following condition:
\begin{align*}
  w(s,\lambda,Z_0) &= 
  \begin{cases}
   >0 & \text{if } \lambda \in \bbR \oplus i (0,1],\\
   0 & \text{if } \lambda \in \bbR,\\
   <0 & \text{if } \lambda \in \bbR \oplus i [-1,0).
  \end{cases}
\end{align*}
Since $\frac{d}{dt}(\Ei(\lambda t)- \Ei(\lambda)) = \frac{e^{\lambda t}}{t}$,
the imaginary part of $\Ei(\lambda t)- \Ei(\lambda)$ is the same sign as $\Imm \lambda$ for $t \in (1,1+\varepsilon_{\lambda})$.
Then, the function $w_1(s,\lambda,Z_0)(\Ei(\lambda t)- \Ei(\lambda))$ is positive for $t \in (1,1+\varepsilon_{\lambda})$.

\begin{thm}\label{thm:startinboundary}
  Let $\sigma_0 = (Z_{B,f^*\omega_Y},\aA_{B,f^*\omega_Y})$ be a type $(C_{-1})$ stability condition in $\partial U(X) \cap \partial U(\rR_{0})$. 
  Fix $(s,\lambda) \in (-\frac{3}{2}-B.C,-\frac{1}{2}-B.C) \times(\bbR_{>0} \oplus i[-1,1])$.
  Let $\uU = \left\{ (s,\lambda,t)\colon t \geq T_0 \right\}.$
  Consider the path of central charges $Z_t$ defined by the equation (\ref{eq:cc}) for $(s,\lambda,t) \in \uU$.
  Then, the path of central charges $Z_t$ for all $(s,\lambda,t) \in \uU$  lifts to  the path of stability conditions $\sigma_{\bullet}$ .
  Moreover, if $\lambda \not = 0$, then the path $\sigma_{\bullet}$ is quasi-convergent.
  \begin{enumerate}
    \item In the case of $\Imm \lambda <0$ the path induces a semiorthogonal decomposition $\langle \oO_C(-1),\dD^b(Y) \rangle$.
    \item In the case of $\Imm \lambda >0$ the path induces a semiorthogonal decomposition $\langle \dD^b(Y),\oO_C \rangle$.
\end{enumerate}
\end{thm}
\begin{proof}
  Set $Z_0 = Z_{B,f^*\omega_Y}$ and $\aA_0 = \aA_{B,f^*\omega_Y}$
  First, we consider $Z_s = \zZ(s,\lambda,Z_0,T_0)$ for $s \in (-\frac{3}{2}-B.C,-\frac{1}{2}-B.C)$.
  The central charge $Z_s$ is the same as the central charge $Z_{B-sC,f^*\omega_Y}$.
  By \Cref{thm:nondegenerate}, the family of central charges $Z_s$ lifts to the path of stability conditions $\sigma_{B-sC,f^*\omega_Y}$.
  It follows from \Cref{thm:deformation} that there exists a real number $\varepsilon$ such that $\zZ(s,\lambda,Z_0,t)$ lifts to $\sigma_{s,\lambda,t}\in \Stab(X)$ for $T_0<t<T_0 + \varepsilon$.
  By the construction of $w(s,\lambda,Z_0)$, $\sigma_t$ lies in $U(\rR_{0})$ for some $t \in (T_0,T_0 + \varepsilon)$.
  Then, the statement follows from \Cref{thm:summarizing}.
\end{proof}

\begin{rmk}
  We can not take $T_0$ uniformly for all $(s,\lambda)$ since $T_0$ might diverges to $\infty$ as $\Imm \lambda$ approaches to $0$.
  However, if we fix $\lambda$ , then we can take $T_0$ uniformly for all $s \in (-\frac{3}{2}-B.C,-\frac{1}{2}-B.C)$.
\end{rmk}

We have a quasi-convergent path starting from the geometric chamber.
\begin{thm}\label{thm:startgeom}
There exists a path of  stability conditions $\sigma_{\bullet}$ with the following conditions.
\begin{enumerate}
  \item The initial value $\sigma_0$ lies in the geometric chamber $U(X)$.
  \item The path of central charges $Z_t$ satisfies the quantum differential equation (\ref{eq:debl}).
  \item The path $\sigma_{\bullet}$ is  quasi-convergent.
\end{enumerate}
\end{thm}
\begin{proof}
  Let $\sigma_{\bullet}\colon [T_0,\infty) \rightarrow \Stab(X)$ be a quasi-convergent path constructed in \Cref{thm:startinboundary}.
  Assume that the parameter $\lambda$ satisfies $\Imm \lambda \not = 0$, and $T_0$ is positive.
  We consider the path of central charges $Z_t$ defined by the equation (\ref{eq:cc}).
  There exists an open neighborhood $V \subset \Hom(H^*(X;\bbC),\bbC)$ of $Z_{T_0}$ and 
  an open neighborhood $\vV \subset \Stab(X)$ of $\sigma_{T_0}$ such that 
  the restricted map $\pi \colon \vV \rightarrow V$ is isomorphic.
  There exists a positive number $\varepsilon$ small enough so that the solution $Z_t$ of the differential equation (\ref{eq:debl}) can be extended for $t \in [T_0 -\varepsilon, \infty)$.
  We assume that $Z_t$ lies in $V$ for $t \in [T_0 -\varepsilon, T_0]$.
  Then, by using the isomorphism $\pi \colon \vV \rightarrow V$, 
  the path of central charges $Z_t$ lifts to the path of stability conditions $\sigma_t$ for $t \in [T_0 -\varepsilon, \infty)$.
  By taking $\varepsilon$ small enough, we have  $Z_t(\oO_C(-1)) \in \bbH$.
  Therefore, the lift $\sigma_t$ is geometric for $t \in [T_0 -\varepsilon, T_0)$.
\end{proof}

Next, we discuss the third question in \Cref{question} when the initial value lies in the gluing region $U(\rR_{0})$.
We define the set $W_2$ to be 
\[
W_2 = \left\{\sigma \in U(\rR_{0}) \colon \phi_{\sigma}(\oO_C) = 2 \right\}.
\]

\begin{lem}\label{lem:T0}
Let $M$ and $\varepsilon$ be positive numbers.
Consider $\vV = [1,M] \times i[-1,1]$.
Define $g(\lambda,t) = \Ei(\lambda t )$.
Then, there exists a positive number $T_0$ 
with $|g(\lambda,t)| > \varepsilon$ for all $\lambda \in \vV$ and $t \geq T_0$.
\end{lem}
\begin{proof}

  By \Cref{lem:ei}, for each $\lambda \in \vV$ there exists a positive number $T_{\lambda}$ such that $|g(\lambda,t)| > \varepsilon$ for  $t \geq T_{\lambda}$.
  We can take $T_{\lambda}$ as a continuous function of $\lambda$.
  Since $\vV$ is compact, there exists a positive number $T_0$ such that $|g(\lambda,t)| > \varepsilon$ for all $\lambda \in \vV$ and $t \geq T_0$.
\end{proof}

Fix $T_0$ as in the above lemma.
Consider $\arg \Ei(\lambda t)$ for $t \geq T_0$ and $\Imm \lambda >0$.
Since $\arg \Ei(\lambda t)$ is a monotonically increasing function for $t$ large enough,
there exists the minimal value of $\arg g(\lambda,t)$ for $t \geq T_0$.
Define the function $w(s,\lambda,Z_0)$ for $\Imm \lambda$ such that the following condition holds:
\[
\arg \left(w(s,\lambda,Z_0) \left(\Ei(\lambda t) - \Ei(\lambda T_0) \right)\right) > -1.
\]
Similarly, we define the function $w(s,\lambda,Z_0)$ for $\Imm \lambda <0$ with 
$\arg \left( w(s,\lambda,Z_0) \left(\Ei(\lambda t) - \Ei(\lambda T_0) \right) \right) < 1.$

\begin{thm}\label{thm:startgluing}
  Let $\sigma_0 = (Z_0,\aA_0)$ be a stability condition in $W_2$.
  Fix positive numbers $M$ and $N$.
  Set $r = Z_0(\oO_C)$.
  Consider a set
  \[
  \uU = (-r,N] \times ([1,M] \oplus i[-1,1]) \times [T_0,\infty).
  \]
  Here, $T_0$ is a positive number large enough such that $\zZ(s,\lambda,Z_0,t) (\oO_C)$ and $\zZ(s,\lambda,Z_0,t) (\oO_C(-1))$ are non-zero for all $t \geq T_0$.
  Consider the path of central charges $Z_t$ defined by the equation (\ref{eq:cc}) for $(s,\lambda,t) \in \uU$.
  Then, the path of central charges $Z_t$ for all $(s,\lambda,t) \in \uU$  lifts to the path of stability conditions $\sigma_{\bullet}$.
  Moreover, if $\Imm \lambda \not = 0$ then the path $\sigma_{\bullet}$ is quasi-convergent.
  \begin{enumerate}
    \item In the case of $\Imm \lambda <0$ the path induces a semiorthogonal decomposition $\langle \oO_C(-1),\dD^b(Y) \rangle$.
  \item  In the case of $\Imm \lambda >0$ the path induces a semiorthogonal decomposition $\langle \dD^b(Y),\oO_C \rangle$.
\end{enumerate}
\end{thm}
\begin{proof}
  First, we consider $Z_s = \zZ(s,\lambda,Z_0,T_0)$ for $s \in (-r,N)$.
  Since $Z_s(\oO_C) >0$, the pair $(Z_s,\aA_0)$ is a stability condition in $W_2$.
  Fix $(s,\lambda) \in (-r,N) \times ([1,M] \oplus i(0,1])$.
  Consider the function $g(t) = \Ei(\lambda t) - \Ei(\lambda T_0)$ for $t \geq T_0$.
  Take $T_{1,\lambda,s}$ large enough such that $\arg \zZ(s,\lambda,Z_0,t) (\oO_C)$ is a monotonically increasing function for $t \geq T_{1,\lambda,s}$.
  By the  construction of $w(s,\lambda,Z_0)$, $\sigma_t$ lies in $U(\rR_{0}) \cup U(\lL_{-1})$ for all $t \in (T_0,T_{1,\lambda,s})$.
  Then, the statement follows from \cref{thm:summarizing}.
  The case of $\Imm \lambda <0$ is similar.
\end{proof}

\newcommand{\etalchar}[1]{$^{#1}$}
\providecommand{\bysame}{\leavevmode\hbox to3em{\hrulefill}\thinspace}
\providecommand{\MR}{\relax\ifhmode\unskip\space\fi MR }
\providecommand{\MRhref}[2]{%
  \href{http://www.ams.org/mathscinet-getitem?mr=#1}{#2}
}
\providecommand{\href}[2]{#2}

\end{document}